\documentclass[a4paper,10pt]{article}
\usepackage{amsmath,amsfonts,amsthm}
\usepackage{a4wide}
\usepackage{booktabs}
\usepackage{cite}
\usepackage{enumerate}
\usepackage{general}
\usepackage{graphicx}
\usepackage{hyperref}
\usepackage[bold]{hhtensor}
\usepackage{nicefrac}
\usepackage{paralist}
\usepackage{xcolor}
\usepackage{xspace}
\usepackage{authblk}
\usepackage{parskip}
\usepackage{pgfplots,pgfplotstable}
\usepackage{subcaption}

\newcommand{\email}[1]{\href{mailto:#1}{#1}}

\usetikzlibrary{external}
\usepackage[normalem]{ulem}
\makeatletter
\def\thm@space@setup{%
  \thm@preskip=\parskip \thm@postskip=0pt
}
\makeatother

\newcommand{\UT}[1][k]{\underline{\mathsf{U}}_T^{#1}}
\newcommand{\Uh}[1][]{\underline{\mathsf{U}}_{h#1}^k}
\newcommand{\UhD}[1][0]{\underline{\mathsf{U}}_{h,#1}^k}

\newcommand{\su}[1][T]{\underline{\mathsf{u}}_{#1}}
\newcommand{\sv}[1][T]{\underline{\mathsf{v}}_{#1}}

\newcommand{\se}[1][T]{\underline{\mathsf{e}}_{#1}}

\newcommand{\unv}[1][]{\mathsf{v}_{#1}}

\newcommand{\bphi}{\vec{\phi}}

\newcommand{\shu}[1][]{\underline{\widehat{\mathsf{u}}}_{#1}}

\newcommand{\pT}[1][k+1]{p_T^{#1}}
\newcommand{\dTF}[1][k]{\delta_{TF}^{#1}}
\newcommand{\ph}[1][k+1]{p_h^{#1}}

\newcommand{\proj}[2][]{\Pi_{#2}^{#1}}
\newcommand{\lproj}[2][h]{\pi_{#1}^{0,#2}}
\newcommand{\eproj}[2][U]{\pi_{#1}^{1,#2}}

\newcommand{\IT}[1][k]{\underline{\mathsf{I}}_T^{#1}}
\newcommand{\Ih}[1][k]{\underline{\mathsf{I}}_h^{#1}}

\newcommand{\GT}[1][k]{\vec{G}^{#1}_{T}}
\newcommand{\Gh}[1][k]{\vec{G}^{#1}_{h}}

\def\bfa{{\mathbf{a}}}

\def\upa{\beta_\bfa}
\def\coera{\lambda_\bfa}
\def\lipa{\gamma_\bfa}
\def\mona{\zeta_\bfa}
\def\div{\mathop{\rm div}}

\def\dx{\,\mathrm{d}\vec{x}}

\newcommand{\asch}{\mathrm{A}_h}
\newcommand{\ascT}{\mathrm{A}_T}
\newcommand{\Eh}{\mathrm{E}_h}
\newcommand{\cE}[1][h]{\mathcal{E}_{#1}}

\newcommand{\argmin}{\operatornamewithlimits{arg\,min}}

\newcommand{\logLogSlopeTriangle}[5]
{

    \pgfplotsextra
    {
        \pgfkeysgetvalue{/pgfplots/xmin}{\xmin}
        \pgfkeysgetvalue{/pgfplots/xmax}{\xmax}
        \pgfkeysgetvalue{/pgfplots/ymin}{\ymin}
        \pgfkeysgetvalue{/pgfplots/ymax}{\ymax}

        \pgfmathsetmacro{\xArel}{#1}
        \pgfmathsetmacro{\yArel}{#3}
        \pgfmathsetmacro{\xBrel}{#1-#2}
        \pgfmathsetmacro{\yBrel}{\yArel}
        \pgfmathsetmacro{\xCrel}{\xArel}

        \pgfmathsetmacro{\lnxB}{\xmin*(1-(#1-#2))+\xmax*(#1-#2)} 
        \pgfmathsetmacro{\lnxA}{\xmin*(1-#1)+\xmax*#1} 
        \pgfmathsetmacro{\lnyA}{\ymin*(1-#3)+\ymax*#3} 
        \pgfmathsetmacro{\lnyC}{\lnyA+#4*(\lnxA-\lnxB)}
        \pgfmathsetmacro{\yCrel}{\lnyC-\ymin)/(\ymax-\ymin)} 

        \coordinate (A) at (rel axis cs:\xArel,\yArel);
        \coordinate (B) at (rel axis cs:\xBrel,\yBrel);
        \coordinate (C) at (rel axis cs:\xCrel,\yCrel);

        \draw[#5]   (A)-- node[pos=0.5,anchor=north] {\scriptsize{1}}
                    (B)-- 
                    (C)-- node[pos=0.,anchor=west] {\scriptsize{#4}} 
                    cycle;
    }
}

\title{$W^{s,p}$-approximation properties of elliptic projectors on polynomial spaces, with application to the error analysis of a Hybrid High-Order discretisation of Leray--Lions problems\footnote{This work was partially supported by ANR project HHOMM (ANR-15-CE40-0005)}}
\author[1]{Daniele A. Di Pietro\footnote{\email{daniele.di-pietro@umontpellier.fr}}}
\affil[1]{University of Montpellier, Institut Montpelli\'{e}rain Alexander Grothendieck, 34095 Montpellier, France}
\author[2]{J\'{e}r\^{o}me Droniou\footnote{\email{jerome.droniou@monash.edu}}}
\affil[2]{School of Mathematical Sciences, Monash University, Clayton, Victoria 3800, Australia}

\begin{document}

\maketitle

\begin{abstract}
  In this work we prove optimal $W^{s,p}$-approximation estimates (with $p\in[1,+\infty]$) for elliptic projectors on local polynomial spaces.
  The proof hinges on the classical Dupont--Scott approximation theory together with two novel abstract lemmas:
  An approximation result for bounded projectors, and an $L^p$-boundedness result for $L^2$-orthogonal projectors on polynomial subspaces.
  The $W^{s,p}$-approximation results have general applicability to (standard or polytopal) numerical methods based on local polynomial spaces.
  As an illustration, we use these $W^{s,p}$-estimates to derive novel error estimates for a Hybrid High-Order discretization of Leray--Lions elliptic problems whose weak formulation is classically set in $W^{1,p}(\Omega)$ for some $p\in(1,+\infty)$.
  This kind of problems appears, e.g., in the modelling of glacier motion, of incompressible turbulent flows, and in airfoil design.
  Denoting by $h$ the meshsize, we prove that the approximation error measured in a $W^{1,p}$-like discrete norm scales as $h^{\frac{k+1}{p-1}}$ when $p\ge 2$ and as $h^{(k+1)(p-1)}$ when $p<2$.


\end{abstract}
%
%
\section{Introduction}

In this work we prove optimal $W^{s,p}$-approximation properties for elliptic projectors on local polynomial spaces, and use these results to derive novel a priori error estimates for a Hybrid High-Order (HHO) discretisation of Leray--Lions elliptic equations.

Let $U\subset\Real^d$, $d\ge 1$, be an open bounded connected set of diameter $h_U$.
For all integers $s\in\Natural$ and all reals $p\in[1,+\infty]$, we denote by $W^{s,p}(U)$ the space of functions having derivatives up to degree $s$ in $L^p(U)$ with associated seminorm
\begin{equation}\label{eq:Wsp.norm}
  \seminorm[W^{s,p}(U)]{v}\eqbydef\sum_{\vec{\alpha}\in\Natural^d,\norm[1]{\vec{\alpha}}=s}\norm[L^p(U)]{\partial^{\vec{\alpha}}v},
\end{equation}
where $\norm[1]{\vec{\alpha}}\eqbydef\alpha_1+\ldots+\alpha_d$ and $\partial^{\vec{\alpha}}=\partial_1^{\alpha_1}\ldots\partial_d^{\alpha_d}$ (this choice for the seminorm enables a seamless treatment of the case $p=+\infty$).

Let a polynomial degree $l\ge 0$ be fixed, and denote by $\Poly{l}(U)$ the space of $d$-variate polynomials on $U$.
The elliptic projector $\eproj{l}:W^{1,1}(U)\to\Poly{l}(U)$ maps a generic function $v\in W^{1,1}(U)$ on the unique polynomial $\eproj{l}v\in\Poly{l}(U)$ obtained in the following way: We start by imposing
\begin{subequations}\label{eq:eproj}
\begin{equation}\label{eq:eproj:1}
\int_U\GRAD\eproj{l}v\SCAL\GRAD w=\int_U\GRAD v\SCAL\GRAD w\qquad\forall w\in\Poly{l}(U).
\end{equation}
By the Riesz representation theorem in $\GRAD\Poly{l}(U)$ for the $L^2(U)^d$-inner product,
this relation defines a unique element $\GRAD\eproj{l}v$, and thus a polynomial
$\eproj{l}v$ up to an additive constant. This constant is then fixed by writing
\begin{equation}\label{eq:eproj:2}
\int_U(\eproj{l} v-v)=0.
\end{equation}
\end{subequations}

We have the following characterisation:
$$
\eproj{l} v = \argmin_{w\in\Poly{l}(U),\,\int_U(w-v)=0}\norm[L^2(U)^d]{\GRAD(w-v)}^2.
$$
The first main result of this work is summarised in the following theorem.

\begin{theorem}[$W^{s,p}$-approximation for $\eproj{l}$]\label{thm:Wsp.approx}
  Assume that $U$ is star-shaped with respect to every point in a ball of radius $\varrho h_U$ for some $\varrho>0$.
  Let $s\in\{1,\ldots,l+1\}$ and $p\in[1,+\infty]$.
  Then, there exists a real number $C>0$ depending only on $d$, $\varrho$, $l$, $s$, and $p$ such that, for all $m\in\{0,\ldots,s\}$ and all $v\in W^{s,p}(U)$,
  \begin{equation}\label{eq:Wsp.approx}
    \seminorm[W^{m,p}(U)]{v-\eproj{l}v}\le C h_U^{s-m}\seminorm[W^{s,p}(U)]{v}.
  \end{equation}
\end{theorem}

The proof of Theorem \ref{thm:Wsp.approx}, given in Section~\ref{sec:Wsp.approx}, is based on the classical Dupont--Scott approximation theory\cite{Dupont.Scott:80} (cf.~also Chapter 4 in Ref.~\cite{Brenner.Scott:08}) and hinges on two novel abstract lemmas for projectors on polynomial spaces: A  $W^{s,p}$-approximation result for projectors that satisfy a suitable boundedness property, and an $L^p$-boundedness result for $L^2$-orthogonal projectors on polynomial subspaces.
Both results make use of the reverse Lebesgue and Sobolev embeddings for polynomial functions proved in Ref.~\cite{Di-Pietro.Droniou:16} (cf., in particular, Lemma 5.1 and Remark A.2 therein).
Following similar arguments as in Section 7 of Ref.~\cite{Dupont.Scott:80}, the results of Theorem \ref{thm:Wsp.approx} still hold if $U$ is a finite union of domains that are star-shaped with respect to balls of radius comparable to $h_U$.

The second main result concerns the approximation of traces, and therefore requires more assumptions on the domain $U$.

\begin{theorem}[$W^{s,p}$-approximation of traces for $\eproj{l}$]\label{thm:Wsp.approx.trace}
  Assume that $U$ is a polytope which admits a partition $\mathcal{S}_U$ into disjoint simplices $S$ of diameter $h_S$ and inradius $r_S$, and that there exists a real number $\varrho>0$ such that, for all $S\in\mathcal{S}_U$,
  $$
  \varrho^2 h_U\le \varrho h_S\le r_S.
  $$
  Let $s\in\{1,\ldots,l+1\}$, $p\in[1,+\infty]$, and denote by $\Fh[U]$ the set of hyperplanar faces of $U$.
  Then, there exists a real number $C$ depending only on $d$, $\varrho$, $l$, $s$ and $p$ such that, for all $m\in \{0,\ldots,s-1\}$ and all $v\in W^{s,p}(U)$,
  \begin{equation}\label{eq:Wsp.approx.trace}
    h_U^{\frac1p}\seminorm[{W^{m,p}(\Fh[U])}]{v-\eproj[U]{l}v}\le Ch_U^{s-m}\seminorm[W^{s,p}(U)]{v}.
  \end{equation}
  Here, $W^{m,p}(\Fh[U])$ denotes the set of functions that belong to $W^{m,p}(F)$ for all $F\in\Fh[U]$, and $\seminorm[{W^{m,p}(\Fh[U])}]{{\cdot}}$ the corresponding broken seminorm.
\end{theorem}

The proof of Theorem \ref{thm:Wsp.approx.trace}, given in Section~\ref{sec:Wsp.approx.trace}, is obtained combining the results of Theorem \ref{thm:Wsp.approx} with a continuous $L^p$-trace inequality.

The approximation results of Theorems \ref{thm:Wsp.approx} and \ref{thm:Wsp.approx.trace} are used to prove novel error estimates for the HHO method of Ref.~\cite{Di-Pietro.Droniou:16} for nonlinear Leray--Lions elliptic problems of the form: Find a potential $u:\Omega\to\Real$ such that
\begin{equation}\label{pde:plap}
  \begin{alignedat}{2}
    -\div(\bfa(\vec{x},\GRAD u))&=f&\qquad&\text{in $\Omega$},
    \\
    u&=0 &\qquad&\text{on $\partial\Omega$},
  \end{alignedat}
\end{equation}
where $\Omega$ is a bounded polytopal subset of $\Real^d$ with boundary $\partial\Omega$, while the source term $f:\Omega\to\Real$ and the function $\bfa:\Omega\times\Real^d\to\Real^d$ satisfy the requirements detailed in Eq. \eqref{assum:gen} below.
Throughout the paper, it is assumed that $\Omega$ does not have cracks, that is,
$\Omega$ lies on one side of its boundary. The family of problems~\eqref{pde:plap},
which contains the $p$-Laplace equation as a special case (cf. \eqref{choice:plap} below), appears in the modelling of glacier motion\cite{Glowinski.Rappaz:03}, of incompressible turbulent flows in porous media\cite{Diaz.Thelin:94}, and in airfoil design\cite{Glowinski:84}.

In the context of conforming Finite Element (FE) approximations of problems which can be traced back to the general form \eqref{pde:plap}, a priori error estimates were derived in Ref.~\cite{Barrett.Liu:94,Glowinski.Rappaz:03}. For nonconforming (Crouzeix--Raviart) FE approximations, error estimates are proved in Ref.~\cite{Liu.Yan:01}, with convergence rates consistent with the ones presented in this work (concerning the link between the HHO method and nonconforming FE, cf. Remark 1 in Ref.~\cite{Di-Pietro.Ern.ea:16} and also Ref.~\cite{Boffi.Di-Pietro:16}).
Error estimates for a nodal Mimetic Finite Difference (MFD) method for a particular kind of operator $\bfa$ and with $p=2$ are proved in Ref.~\cite{Antonietti.Bigoni.ea:14}.
Finite volume methods, on the other hand, are considered in Ref.~\cite{Andreianov.Boyer.ea:06}, where error estimates similar to the ones obtained here are derived under the assumption that the source term $f$ vanishes on the boundary (additional error terms are present when this is not the case).
Finally, we also cite here Ref.~\cite{Droniou:06}, where the convergence study of a Mixed Finite Volume (MFV) scheme inspired by Ref.~\cite{Droniou.Eymard:06} is carried out using a compactness argument under minimal regularity assumptions on the exact solution.

The HHO method analysed here is based on meshes composed of general polytopal elements and its formulation hinges on degrees of freedom (DOFs) that are polynomials of degree $k\ge 0$ on mesh elements and faces; cf. Refs.~\cite{Di-Pietro.Ern.ea:14,Di-Pietro.Ern:15,Di-Pietro.Droniou.ea:15,Di-Pietro.Ern:16} for an introduction to HHO methods and and Refs.~\cite{Di-Pietro.Droniou:16,Chave.Di-Pietro.ea:16} for applications to nonlinear problems.
Based on such DOFs, a gradient reconstruction operator $\GT$ of degree $k$ and a potential reconstruction operator $\pT$ of degree $(k+1)$ are devised by solving local problems inside each mesh element $T$.
By construction, the composition of the potential reconstruction $\pT$ with the interpolator on the DOF space coincides with the elliptic projector $\eproj[T]{k+1}$.
The gradient and potential reconstruction operators are then used to formulate a local contribution composed of a consistent and a stabilisation term.
The $W^{s,p}$-approximation properties for $\eproj[T]{k+1}$ play a crucial role in estimating the error associated with the latter.
Denoting by $h$ the meshsize, we prove in Theorem \ref{thm:est.error} below that, for smooth enough exact solutions, the approximation error measured in a discrete $W^{1,p}$-like norm converges as $h^{\frac{k+1}{p-1}}$ when $p\ge 2$ and as $h^{(k+1)(p-1)}$ when $1<p<2$.
A detailed comparison with the literature is provided in Remark \ref{rem:order.of.conv}.

As noticed in Ref.~\cite{Di-Pietro.Ern.ea:14}, the lowest-order version of the HHO method corresponding to $k=0$ is essentially analogous (up to equivalent stabilisation) to the SUSHI scheme of Ref.~\cite{Eymard.Gallouet.ea:10} when face unknowns are not eliminated by interpolation. This method, in turn, has been proved in Ref.~\cite{Droniou.Eymard.ea:10} to be
equivalent to the MFV method of Ref.~\cite{Droniou.Eymard:06} and the mixed-hybrid MFD method of Ref.~\cite{Kuznetsov.Lipnikov.ea:04,Brezzi.Lipnikov.ea:05} (cf. also Ref.~\cite{Beirao-da-Veiga.Lipnikov.ea:14} for an introduction to MFD methods). As a consequence, our results extend the analysis
conducted in Ref.~\cite{Droniou:06}, by providing in particular error estimates for the MFV scheme applied to Leray--Lions equations.

To conclude, it is worth mentioning that the tools of Theorems \ref{thm:Wsp.approx} and \ref{thm:Wsp.approx.trace}, alongside the optimum $W^{s,p}$-estimates of Ref.~\cite{Di-Pietro.Droniou:16} for $L^2$-projectors on polynomial spaces (see Lemma \ref{lem:Wkp.interp}), are potentially of interest also for the study of other polytopal methods.
Elliptic projections on polynomial spaces appear, e.g., in the conforming and nonconforming Virtual Element Methods (cf. Eq. (4.18) in Ref.~\cite{Beirao-da-Veiga.Brezzi.ea:13} and Eqs. (3.18)--(3.20) in Ref.~\cite{Ayuso-de-Dios.Lipnikov.ea:16}, respectively).
They also play a role in determining the high-order part of some post-processings of the potential used in the context of Hybridizable Discontinuous Galerkin methods; cf., e.g., the variation proposed in Ref.~\cite{Cockburn.Di-Pietro.ea:16} of the post-processing considered in Refs.~\cite{Cockburn.Gopalakrishnan.ea:10,Cockburn.Qiu.ea:12}.

The rest of the paper is organised as follows.
In Section \ref{sec:eproj} we provide the proofs of Theorems \ref{thm:Wsp.approx} and \ref{thm:Wsp.approx.trace} preceeded by the required preliminary results.
In Section \ref{sec:appl} we use these results to derive error estimates for the HHO discretization of problem \ref{pde:plap}.
\ref{app:llop} collects some useful inequalities for Leray--Lions operators.
%
%
\section{$W^{s,p}$-approximation properties of the elliptic projector on polynomial spaces}\label{sec:eproj}

This section contains the proofs of Theorems \ref{thm:Wsp.approx} and \ref{thm:Wsp.approx.trace} preceeded by two abstract lemmas for projectors on polynomials subspaces.
Throughout the paper, to alleviate the notation, when writing integrals we omit the dependence on the integration variable $\vec{x}$ as well as the differential with the exception of those integrals involving the function $\bfa$ (cf.~\eqref{pde:plap}).

\subsection{Two abstract results for projectors on polynomial subspaces}\label{sec:eproj:abstract.results}

Our first lemma is an abstract approximation result valid for any projector on a polynomial space that satisfies a suitable boundedness property.

\begin{lemma}[$W^{s,p}$-approximation for $W$-bounded projectors]\label{lem:Wsp.approx.abstract}
  Assume that $U$ is star-shaped with respect to every point of a ball of radius $\varrho h_U$ for some $\varrho>0$.
  Let a real number $p\in[1,+\infty]$ and four integers $l\ge 0$, $s\in\{1,\ldots,l+1\}$, and $q,m\in\{0,\ldots,s\}$ be fixed.
  Let $\proj[q,l]{U}:W^{q,1}(U)\to\Poly{l}(U)$ be a projector such that there exists a real number $C>0$ depending only on $d$, $\varrho$, $l$, $q$, and $p$ such that for all $v\in W^{q,p}(U)$,
  \begin{subequations}\label{eq:Wq1.boundedness}
  \begin{align}\label{eq:Wq1.boundedness.1}
    &\mbox{If $m<q$}:\quad\seminorm[W^{m,p}(U)]{\proj[q,l]{U} v}\le C\sum_{r=m}^qh_U^{r-m}\seminorm[W^{r,p}(U)]{v},\\
		\label{eq:Wq1.boundedness.2}
    &\mbox{If $m\ge q$}:\quad\seminorm[W^{q,p}(U)]{\proj[q,l]{U} v}\le C\seminorm[W^{q,p}(U)]{v},
  \end{align}
  \end{subequations}
  Then, there exists a real number $C>0$ depending only on $d$, $\varrho$, $l$, $q$, $m$, $s$, and $p$ such that, for all $v\in W^{s,p}(U)$,
  \begin{equation}\label{eq:Wsp.approx.abstract}
    \seminorm[W^{m,p}(U)]{v-\proj[q,l]{U}v}\le C h_U^{s-m}\seminorm[W^{s,p}(U)]{v}.
  \end{equation}
\end{lemma}

\begin{proof}
  Here $A\lesssim B$ means $A\le MB$ with real number $M>0$ having the same dependencies as $C$ in \eqref{eq:Wsp.approx.abstract}.
  Since smooth functions are dense in $W^{s,p}(U)$, we can assume $v\in C^\infty(U)\cap W^{s,p}(U)$.
  We consider the following representation of $v$, proposed in Chapter 4 of Ref.~\cite{Brenner.Scott:08}:
  \begin{equation}\label{sob.rep}
    v=Q^s v + R^s v,
  \end{equation}
  where $Q^sv\in\Poly{s-1}(U)\subset\Poly{l}(U)$ is the averaged Taylor polynomial, while the remainder $R^sv$ satisfies, for all $r\in\{0,\ldots,s\}$ (cf. Lemma 4.3.8 in Ref.~\cite{Brenner.Scott:08}),
  \begin{equation}\label{sob.rep.rem}
    \seminorm[W^{r,p}(U)]{R^s v}\lesssim h_U^{s-r}\seminorm[W^{s,p}(U)]{v}.
  \end{equation}
  Since $\proj[q,l]{U}$ is a projector, it holds $\proj[q,l]{U}(Q^sv)=Q^sv$ so that, taking the projection of~\eqref{sob.rep}, it is inferred
  $$
  \proj[q,l]{U}v=Q^sv + \proj[q,l]{U}(R^sv).
  $$
  Subtracting this equation from \eqref{sob.rep}, we arrive at $v-\proj[q,l]{U}v=R^sv-\proj[q,l]{U}(R^sv)$.
  Hence, the triangle inequality yields
  \begin{equation}\label{eq:Wsp.approx.abs:1}
    \seminorm[W^{m,p}(U)]{v-\proj[q,l]{U}v}\le
    \seminorm[W^{m,p}(U)]{R^sv}+ \seminorm[W^{m,p}(U)]{\proj[q,l]{U}(R^sv)}.
  \end{equation}  
  For the first term in the right-hand side, the estimate \eqref{sob.rep.rem} with $r=m$ readily yields
    \begin{equation}\label{eq:Wsp.approx.abs:2}
      \seminorm[W^{m,p}(U)]{R^sv}\lesssim h_U^{s-m}\seminorm[W^{s,p}(U)]{v}.
    \end{equation}
    Let us estimate the second term.
    If $m<q$, using the boundedness assumption \eqref{eq:Wq1.boundedness.1} followed by the estimate \eqref{sob.rep.rem}, it is inferred
    $$
    \begin{aligned}
    \seminorm[W^{m,p}(U)]{\proj[q,l]{U}(R^sv)}
      &\lesssim\sum_{r=m}^qh_U^{r-m}\seminorm[W^{r,p}(U)]{R^sv} \\
      &\lesssim\sum_{r=m}^qh_U^{r-m} h_U^{s-r}\seminorm[W^{s,p}(U)]{v}
    \lesssim h_U^{s-m}\seminorm[W^{s,p}(U)]{v}.
    \end{aligned}
      $$
  If, on the other hand, $m\ge q$, using the reverse Sobolev embeddings on polynomial spaces of Remark A.2 in Ref.~\cite{Di-Pietro.Droniou:16} followed by assumption~\eqref{eq:Wq1.boundedness.2} and the estimate \eqref{sob.rep.rem} with $r=q$, it is inferred that
  $$
  \begin{aligned}
    \seminorm[W^{m,p}(U)]{\proj[q,l]{U}(R^sv)}
    &\lesssim h_U^{q-m}\seminorm[W^{q,p}(U)]{\proj[q,l]{U}(R^sv)} \\
    &\lesssim h_U^{q-m}\seminorm[W^{q,p}(U)]{R^sv}
    \lesssim h_U^{s-m}\seminorm[W^{s,p}(U)]{v}.
  \end{aligned}
    $$
    In conclusion we have, in either case $m<q$ or $m\ge q$,
    \begin{equation}\label{eq:Wsp.approx.abs:3}
      \seminorm[W^{m,p}(U)]{\proj[q,l]{U}(R^sv)}\lesssim h_U^{s-m}\seminorm[W^{s,p}(U)]{v}.
    \end{equation}
    Using \eqref{eq:Wsp.approx.abs:2} and \eqref{eq:Wsp.approx.abs:3} to estimate the first and second term in the right-hand side of \eqref{eq:Wsp.approx.abs:1}, respectively, the conclusion follows.
\end{proof}

Our second technical result concerns the $L^p$-boundedness of $L^2$-orthogonal projectors on polynomial subspaces, and will be central to prove property \eqref{eq:Wq1.boundedness} (with $q=1$) for the elliptic projector $\eproj{l}$.
This result generalises Lemma 3.2 in Ref.~\cite{Di-Pietro.Droniou:16}, which corresponds to $\mathcal P=\Poly{l}(U)$.

\begin{lemma}[$L^p$-boundeness of $L^2$-orthogonal projectors on polynomial subspaces]\label{lem:Lpstab.gradproj}
  Let two integers $l\ge 0$ and $n\ge 1$ be fixed, and let $\mathcal{P}$ be a subspace of $\Poly{l}(U)^n$.
  We consider the $L^2$-orthogonal projector $\proj{\mathcal{P}}:L^1(U)^n\to\mathcal{P}$ such that, for all $\Phi\in L^1(U)^n$,
  \begin{equation}\label{def:Lproj}
    \int_T (\proj{\mathcal{P}}\Phi-\Phi)\SCAL\Psi = 0\text{ for all $\Psi\in \mathcal{P}$}.
  \end{equation}  
  Let $p\in [1,+\infty]$.
  Let $r_U$ be the inradius of $U$ and assume that there is a real number $\delta$ such that $$\frac{r_U}{h_U}\ge\delta>0.$$
  Then, there exists a real number $C>0$ depending only on $n$, $d$, $\delta$, $l$, and $p$ such that
  \begin{equation}\label{eq:Lp.boundedness.gradproj}
    \forall\Phi\in L^p(U)^n\,:\,\norm[L^p(U)^n]{\proj{\mathcal{P}}\Phi}\le C\norm[L^p(U)^n]{\Phi}.
  \end{equation}
\end{lemma}

\begin{remark}[Dependence of $C$ in~\eqref{eq:Lp.boundedness.gradproj}]
  At least on selected geometries, inequality~\eqref{eq:Lp.boundedness.gradproj} holds with constant $C$ independent of $\delta$.
  Whether this is true in general remains an open question, which possibly requires different techniques than the ones used here to answer.
  In any case, this does not change the fact that the constants appearing in Theorems~\ref{thm:Wsp.approx} and~\ref{thm:Wsp.approx.trace} \emph{do} depend on $\varrho$.
\end{remark}

\begin{proof}
  We abridge as $A\lesssim B$ the inequality $A\le MB$ with real number $M>0$ having the same dependencies as $C$.
  Since $\proj{\mathcal{P}}$ is an $L^2$-orthogonal projector, \eqref{eq:Lp.boundedness.gradproj} trivially holds with $C=1$ if $p=2$.
  On the other hand, if $p>2$, we have, using the reverse Lebesgue embeddings on polynomial spaces of Lemma 3.2 in Ref.~\cite{Di-Pietro.Droniou:16} followed by~\eqref{eq:Lp.boundedness.gradproj} for $p=2$,
  \[
  \norm[L^{p}(U)^n]{\proj{\mathcal{P}}\Phi}
  \lesssim\meas[d]{U}^{\frac1p-\frac12}\norm[L^2(U)^n]{\proj{\mathcal{P}}\Phi}
  \lesssim\meas[d]{U}^{\frac1p-\frac12}\norm[L^2(U)^n]{\Phi}.
  \]
  Here, $\meas[d]{U}$ is the $d$-dimensional measure of $U$.
  Using the H\"{o}lder inequality to infer $\norm[L^2(U)^n]{\Phi}\lesssim\meas[d]{U}^{\frac12-\frac1p}\norm[L^p(U)^n]{\Phi}$ concludes the proof for $p>2$.    
  It only remains to treat the case $p<2$.
  We first observe that, using the definition \eqref{def:Lproj} of $\proj{\mathcal{P}}$ twice, for all
  $\Phi,\Psi\in L^1(U)^n$,
  \[
  \int_U (\proj{\mathcal{P}}\Phi)\SCAL \Psi=	\int_U (\proj{\mathcal{P}}\Phi)\SCAL (\proj{\mathcal{P}}\Psi)=
  \int_U \Phi\SCAL (\proj{\mathcal{P}}\Psi).
  \]
  Hence, with $p'$ such that $\frac{1}{p}+\frac{1}{p'}=1$, it holds
  \begin{equation}\label{eq:Lp.boundedness.gradproj:1}
    \begin{aligned}
      \norm[L^p(U)^n]{\proj{\mathcal{P}}\Phi}
      &= \sup_{\Psi\in L^{p'}(U)^n,\norm[L^{p'}(U)^n]{\Psi}=1} \int_U(\proj{\mathcal{P}}\Phi)\SCAL \Psi
      \\
      &=\sup_{\Psi\in L^{p'}(U)^n,\norm[L^{p'}(U)^n]{\Psi}=1} \int_U \Phi\SCAL (\proj{\mathcal{P}}\Psi)
      \\
      &\le \sup_{\Psi\in L^{p'}(U)^n,\norm[L^{p'}(U)^n]{\Psi}=1} \norm[L^p(U)^n]{\Phi}\norm[L^{p'}(U)^n]{\proj{\mathcal{P}}\Psi},
    \end{aligned}
  \end{equation}
  where we have used the H\"{o}lder inequality to conclude.
  Using \eqref{eq:Lp.boundedness.gradproj} for $p'>2$, we have $\norm[L^{p'}(U)^n]{\proj{\mathcal{P}}\Psi}\lesssim\norm[L^{p'}(U)^n]{\Psi}=1$. Plugging this bound into \eqref{eq:Lp.boundedness.gradproj:1} concludes the proof for $p<2$.
\end{proof}

\subsection{Proof of the main results}

We are now ready to prove Theorems \ref{thm:Wsp.approx} and \ref{thm:Wsp.approx.trace}.
Inside the proofs, $A\lesssim B$ means $A\le MB$ with $M$ having the same dependencies as the real number $C$ in the corresponding statement.

\subsubsection{Proof of Theorem \ref{thm:Wsp.approx}}\label{sec:Wsp.approx}

The proof of~\eqref{eq:Wsp.approx} is obtained applying Lemma \ref{lem:Wsp.approx.abstract} with $q=1$ and $\proj[1,l]{U}=\eproj[U]{l}$. To prove that the condition \eqref{eq:Wq1.boundedness} holds, we distinguish two cases: $m\ge 1$, treated in {\bf Step 1}, and $m=0$, treated in {\bf Step 2}.

  \begin{asparaenum}[\bf Step 1.]
  \item \emph{The case $m\ge 1$.} 
    We need to show that~\eqref{eq:Wq1.boundedness.2} holds, i.e.,
    \begin{equation}\label{eq:W1p.boundedness}  
      \forall v\in W^{1,p}(U)\,:\,  
      \seminorm[W^{1,p}(U)]{\eproj{l} v}\lesssim\seminorm[W^{1,p}(U)]{v}.
    \end{equation}
    By definition \eqref{eq:eproj} of $\eproj[U]{l}$, it holds, for all $v\in W^{1,1}(U)$, 
    \begin{equation}\label{eq:Wsp.approx:1}
      \GRAD\eproj[U]{l} v = \proj{\GRAD\Poly{l}(U)}\GRAD v,
    \end{equation}
    where $\proj{\GRAD\Poly{l}(U)}$ denotes the $L^2$-orthogonal projector on $\GRAD\Poly{l}(U)\subset\Poly{l-1}(U)^{d}$.
    Then, \eqref{eq:W1p.boundedness} is proved observing that, by definition \eqref{eq:Wsp.norm} of the $\seminorm[W^{1,p}(U)]{{\cdot}}$-seminorm, and invoking \eqref{eq:Wsp.approx:1} and the $(L^p)^d$-boundedness of $\proj{\GRAD\Poly{l}(U)}$ resulting from \eqref{eq:Lp.boundedness.gradproj} with $\mathcal{P}=\GRAD\Poly{l}(U)$, we have
    \begin{equation*}
      \seminorm[W^{1,p}(U)]{\eproj{l}v}
      \lesssim\norm[L^p(U)^d]{\GRAD\eproj{l}v}
      =\norm[L^p(U)^d]{\proj{\GRAD\Poly{l}(U)}\GRAD v}
      \lesssim \norm[L^p(U)^d]{\GRAD v}\lesssim \seminorm[W^{1,p}(U)]{v}.
    \end{equation*}

  \item \emph{The case $m=0$.} 
    We need to prove that~\eqref{eq:Wq1.boundedness.1} holds, i.e.,
      \begin{equation}\label{eq:Lp.boundedness}
        \forall v\in W^{1,p}(U)\,:\,
        \norm[L^p(U)]{\eproj{l} v}\lesssim h_U\seminorm[W^{1,p}(U)]{v} + \norm[L^p(U)]{v}.
      \end{equation}
      Let $v\in W^{1,p}(U)$ and denote by $\overline{v}\in\Poly{0}(U)$ the $L^2$-orthogonal projection of $v$ on $\Poly{0}(U)$ such that
      $$
      \text{$\int_U (v-\overline{v})=0$, that is, $\overline{v}=\frac{1}{|U|_d}\int_U v$.}
      $$
      By definition \eqref{eq:eproj} of the elliptic projector, $\overline{v}$ is also the $L^2$-orthogonal projection on $\Poly{0}(U)$ of $\eproj[U]{l}v$. The $W^{s,p}$-approximation of the $L^2$-projector \eqref{eq:approx.lproj.Wsp} (applied with $m=0$ and $s=1$ to $\eproj[U]{l}v$ instead of $v$) therefore gives $\norm[L^p(U)]{\eproj{l} v-\overline{v}}\lesssim h_U\seminorm[W^{1,p}(U)]{\eproj{l} v}$. This yields
      $$
      \begin{aligned}
        \norm[L^p(U)]{\eproj{l} v}
        &\le\norm[L^p(U)]{\eproj{l} v-\overline{v}} + \norm[L^p(U)]{\overline{v}}
        \\
        &\lesssim h_U\seminorm[W^{1,p}(U)]{\eproj{l} v} + \norm[L^p(U)]{\overline{v}}
        \\
        &\lesssim h_U\seminorm[W^{1,p}(U)]{v} + \norm[L^p(U)]{v},
      \end{aligned}
      $$
      where we have introduced $\pm\overline{v}$ inside the norm and used the triangle inequality in the first line, and the terms in the third line are have been estimated using \eqref{eq:W1p.boundedness} for the first one and the  Jensen inequality for the second one.
  \end{asparaenum}
  
\subsubsection{Proof of Theorem \ref{thm:Wsp.approx.trace}}\label{sec:Wsp.approx.trace}

  Under the assumptions on $U$, we have the following $L^p$-trace inequality (cf. Lemma~3.6 in Ref.~\cite{Di-Pietro.Droniou:16} for a proof): For all $w\in W^{1,p}(U)$,
  \begin{equation}\label{ineq.cont.trace}
    h_U^{\frac1p}\norm[L^p(\partial U)]{w}\lesssim \norm[L^p(U)]{w}+h_U\norm[L^p(U)]{\GRAD w}.
  \end{equation}
  For $m\le s-1$, by applying \eqref{ineq.cont.trace} to $w=\partial^{\vec{\alpha}}(v-\eproj[U]{l}v)\in W^{1,p}(U)$ for all $\vec{\alpha}\in\Natural^d$ such that $\norm[1]{\vec{\alpha}}=m$, we find
  \[
  h_U^{\frac1p}\seminorm[{W^{m,p}(\Fh[U])}]{v-\eproj[U]{l}v}\lesssim 
  \seminorm[W^{m,p}(U)]{v-\eproj[U]{l}v}+h_U\seminorm[W^{m+1,p}(U)]{v-\eproj[U]{l}v}.
  \]
  To conclusion follows using~\eqref{eq:Wsp.approx} for $m$ and $m+1$ to bound the two terms in the right-hand side.
%
%
\section{Error estimates for a Hybrid High-Order discretisation of Leray--Lions problems}\label{sec:appl}

In this section we use the approximation results for the elliptic projector to derive new error estimates for the HHO discretisation of Leray--Lions problems introduced in Ref.~\cite{Di-Pietro.Droniou:16} (where convergence to minimal regularity solutions is proved using a compactness argument).

\subsection{Continuous model}

We consider problem \eqref{pde:plap} under the following assumptions for a fixed $p\in (1,+\infty)$ with $p'\eqbydef\frac{p}{p-1}$:
\begin{subequations}
  \label{assum:gen}
  \begin{equation}
    \label{hyp:fg}
    f\in L^{p'}(\Omega),
  \end{equation}
  \begin{equation}
    \label{hyp:acarat}
    \mbox{$\bfa:\Omega\times\Real^d\to\Real^d$ is a Caratheodory function},
  \end{equation}
  \begin{equation}
    \label{hyp:ag}
    \begin{array}{l}
      \bfa(\cdot,\vec{0})\in L^{p'}(\Omega)^d\mbox{ and }
      \exists\upa\in(0,+\infty)\,:\\
      |\bfa(\vec{x},\vec{\xi})-\bfa(\vec{x},\vec{0})|\le \upa|\vec{\xi}|^{p-1}
      \mbox{ for a.e. $\vec{x}\in\Omega$, for all $\vec{\xi}\in\Real^d$},
    \end{array}
  \end{equation}
  \begin{equation}
    \label{hyp:ac}
    \exists \coera\in (0,+\infty)\,:\,\bfa(\vec{x},\vec{\xi})\cdot\vec{\xi} \ge \coera|\vec{\xi}|^p
    \mbox{ for a.e. $\vec{x}\in\Omega$, for all $\vec{\xi}\in\Real^d$},
  \end{equation}
  \begin{equation}
    \label{hyp:alip}
    \begin{array}{l}
      \exists\lipa\in (0,+\infty)\,:\,
      |\bfa(\vec{x},\vec{\xi})-\bfa(\vec{x},\vec{\eta})|\le \lipa |\vec{\xi}-\vec{\eta}|
      (|\vec{\xi}|^{p-2}+|\vec{\eta}|^{p-2})\\
      \qquad\mbox{ for a.e. $\vec{x}\in\Omega$, for all $(\vec{\xi},\vec{\eta})\in\Real^d\times\Real^d$},
    \end{array}
  \end{equation}
  \begin{equation}
    \label{hyp:am}
    \begin{array}{l}
      \exists\mona\in (0,+\infty)\,:\,
      [\bfa(\vec{x},\vec{\xi})-\bfa(\vec{x},\vec{\eta})]\cdot[\vec{\xi}-\vec{\eta}]\ge 
      \mona |\vec{\xi}-\vec{\eta}|^2(|\vec{\xi}|+|\vec{\eta}|)^{p-2}\\
      \qquad\mbox{ for a.e. $\vec{x}\in\Omega$, for all $(\vec{\xi},\vec{\eta})\in\Real^d\times\Real^d$},
    \end{array}
  \end{equation}
\end{subequations}

Assumptions \eqref{hyp:acarat}--\eqref{hyp:ac} are the pillars of Leray--Lions operators and stipulate, respectively, the regularity for $\bfa$, its growth, and its coercivity.
Assumptions \eqref{hyp:alip} and \eqref{hyp:am} additionally require the Lipschitz continuity and uniform monotonicity of $\bfa$ in an appropriate form.

\begin{remark}[$p$-Laplacian]
A particularly important example of Leray--Lions problem is the $p$-Laplace equation, which corresponds to the function
\begin{equation}\label{choice:plap}
  \bfa(\vec{x},\vec{\xi})=|\vec{\xi}|^{p-2}\vec{\xi}.
\end{equation}
Properties \eqref{hyp:acarat}--\eqref{hyp:ac} are trivially verified for this choice, which additionally verifies \eqref{hyp:alip} and \eqref{hyp:am};
cf. Ref.~\cite{Barrett.Liu:94} for a proof of the former and Ref.~\cite{gdm} for a proof of both.
\end{remark}

As usual, problem \eqref{pde:plap} is understood in the following weak sense:
\begin{equation}\label{pde:plap.weak}
	\begin{aligned}
  &\text{Find $u\in W^{1,p}_0(\Omega)$ such that, for all $v\in W^{1,p}_0(\Omega)$,}\\
	&\int_\Omega \bfa(\vec{x},\GRAD u(\vec{x}))\cdot \GRAD v(\vec{x})\dx = \int_\Omega fv,
	\end{aligned}
\end{equation}
where $W^{1,p}_0(\Omega)$ is spanned by the elements of $W^{1,p}(\Omega)$ that vanish on $\partial\Omega$ in the sense of traces.

\subsection{The Hybrid High-Order method}

We briefly recall here the construction of the HHO method and a few known results that will be needed in the analysis.

\subsubsection{Mesh and notations}

Let us start by the notion of mesh, inspired from Definition 7.2 in Ref.~\cite{gdm}, and some associated notations.

\begin{definition}[Mesh and set of faces]\label{def:mesh}
  A mesh $\Th$ of the domain $\Omega$ is a finite collection of nonempty disjoint open polytopal elements $T$ with boundary $\partial T$ and diameter $h_T$
  such that $\closure{\Omega}=\bigcup_{T\in\Th}\closure{T}$ and $h=\max_{T\in\Th} h_T$.

  The set of faces $\Fh$ is a finite family of disjoint subsets of $\overline{\Omega}$
  such that, for any $F\in\Fh$, $F$ is an open subset of a hyperplane of $\Real^d$,
  the $(d{-}1)$-dimensional Hausdorff measure
  of $F$ is strictly positive, and the $(d-1)$-dimensional Hausdorff measure
  of its relative interior $\closure{F}\backslash F$ is zero. The diameter of $F$ is
denoted by $h_F$.
  Additionally,
  \begin{enumerate}[(i)]
  \item For each $F\in \Fh$, either
    \begin{inparaenum}[(a)]%
    \item there exist distinct mesh elements $T_1,T_2\in\Th $ such that $F\subset\partial T_1\cap\partial T_2$ and $F$ is called an interface or 
    \item there exists a mesh element $T\in\Th$ (which is unique since $\Omega$ is assumed to have no cracks) such that $F\subset\partial T\cap\partial\Omega$ and $F$ is called a boundary face.
\end{inparaenum}
  \item The set of faces is a partition of the mesh skeleton:
    $\bigcup_{T\in\Th}\partial T = \bigcup_{F\in\Fh}\closure{F}$.
  \end{enumerate}
  
  For any mesh element $T\in\Th$, $\Fh[T]\eqbydef\{F\in\Fh\st F\subset\partial T\}$ denotes the set of faces contained in $\partial T$.
  For all $F\in\Fh[T]$, $\normal_{TF}$ is the unit normal to $F$ pointing out of $T$.

  Interfaces are collected in the set $\Fhi$, boundary faces in $\Fhb$, and $\Fh=\Fhi\cup\Fhb$.
\end{definition}
\medskip
\begin{remark}[Element and boundary faces]
  As a result of Definition~\ref{def:mesh}, above, it holds that $\partial T=\bigcup_{F\in\Fh[T]}\closure{F}$ for all $T\in\Th$, and that $\partial\Omega=\bigcup_{F\in\Fhb}\closure{F}$.
\end{remark}

Throughout the rest of the paper, we assume the following regularity for $\Th$ inspired by Chapter 1 in Ref.~\cite{Di-Pietro.Ern:12}.

\begin{assumption}[Regularity assumption on $\Th$]
  \label{def:adm.Th}
  The mesh $\Th$ admits a matching simplicial submesh $\fTh$ and there exists a real number $\varrho>0$ such that:
  \begin{inparaenum}[(i)]%
  \item For all simplices $S\in\fTh$ of diameter $h_S$ and inradius $r_S$, $\varrho h_S\le r_S$, and
  \item for all $T\in\Th$, and all $S\in\fTh$ such that $S\subset T$, $\varrho h_T \le h_S$.
  \end{inparaenum}
\end{assumption}

When working on refined mesh sequences, all the (explicit or implicit) constants we consider below remain bounded provided that $\varrho$ remains bounded away from $0$ in the refinement process.
Additionally, mesh elements satisfy the geometric regularity assumptions that enable the use of both Theorems \ref{thm:Wsp.approx} and \ref{thm:Wsp.approx.trace} (as well as Lemma \ref{lem:Wkp.interp} below).

\subsubsection{Degrees of freedom and interpolation operators}

Let a polynomial degree $k\ge 0$ and a mesh element $T\in\Th$ be fixed. The local space of
degrees of freedom (DOFs) is
\begin{equation}
  \label{eq:UT}
  \UT\eqbydef\Poly{k}(T)\times\left(
  \bigtimes_{F\in\Fh[T]}\Poly{k}(F)
  \right),
\end{equation}
where $\Poly{k}(F)$ denotes the space spanned by the restriction to $F$ of $d$-variate polynomials.
We use the underlined notation $\sv=(\unv[T],(\unv[F])_{F\in\Fh[T]})$ for a generic element $\sv\in\UT$.
If $U=T\in\Th$ or $U=F\in\Fh$, we define the $L^2$-projector $\lproj[U]{l}: L^1(U)\to \Poly{l}(U)$ such that,
for any $v\in L^1(U)$, $\lproj[U]{l}v$ is the unique element of $\Poly{l}(U)$ satisfying
\begin{equation}\label{def:lproj}
  \forall w\in\Poly{l}(U)\,:\,
  \int_U (\lproj[U]{l}v-v)~w = 0.
\end{equation}
When applied to vector-valued function, it is understood that $\lproj[U]{l}$ acts component-wise.
The local interpolation operator $\IT: W^{1,1}(T)\to\UT$ is then given by
\begin{equation}
  \label{eq:IT}
  \forall v\in W^{1,1}(T)\,:\,  \IT v \eqbydef (\lproj[T]{k}v, (\lproj[F]{k}v)_{F\in\Fh[T]}).
\end{equation}

Local DOFs are collected in the following global space obtained by patching interface values:
\begin{equation*}
  \Uh\eqbydef\left(
  \bigtimes_{T\in\Th}\Poly{k}(T)
  \right)\times
  \left(
  \bigtimes_{F\in\Fh}\Poly{k}(F)
  \right).
\end{equation*}
A generic element of $\Uh$ is denoted by $\sv[h]=((\unv[T])_{T\in\Th},(\unv[F])_{F\in\Fh})$ and, for all $T\in\Th$, $\sv[T]=(\unv[T],(\unv[F])_{F\in\Fh[T]})$ is its restriction to $T$.
We also introduce the notation $\unv[h]$ for the broken polynomial function in $\Poly{k}(\Th)\eqbydef\left\{v\in L^1(\Omega)\,:\,\restrto{v}{T}\in\Poly{k}(T)\quad\forall T\in\Th\right\}$ obtained from element-based DOFs by setting $\restrto{\unv[h]}{T}=\unv[T]$ for all $T\in\Th$.
The global interpolation operator $\Ih:W^{1,1}(\Omega)\to\Uh$ is such that
\begin{equation}\label{eq:Ih}
  \forall v\in W^{1,1}(\Omega)\,:\,
  \Ih v \eqbydef ( (\lproj[T]{k} v)_{T\in\Th}, (\lproj[F]{k} v)_{F\in\Fh} ).
\end{equation}

\subsubsection{Gradient and potential reconstructions}

For $U=T\in\Th$ or $U=F\in\Fh$, we denote henceforth by $(\cdot,\cdot)_U$ the $L^2$- or $(L^2)^d$-inner product on $U$.
The HHO method hinges on the local discrete gradient operator $\GT:\UT\to \Poly{k}(T)^d$ such that, for all $\sv=(\unv[T],(\unv[F])_{F\in\Fh[T]})\in\UT$, $\GT\sv$ solves the following problem:
For all $\bphi\in \Poly{k}(T)^d$,
\begin{equation}\label{def:GT.bis}    
  (\GT \sv,\bphi)_T
  \eqbydef -(\unv[T],\DIV\bphi)_T
  + \sum_{F\in\Fh[T]} (\unv[F], \bphi\SCAL\normal_{TF})_F.
\end{equation}
Existence and uniqueness of $\GT\sv$ immediately follow from the Riesz representation theorem in $\Poly{k}(T)^d$ for the standard $L^2(T)^d$-inner product. The right-hand side of~\eqref{def:GT.bis} mimicks an integration by parts formula where the role of the scalar function inside volumetric and boundary integrals is played by element-based and face-based DOFs, respectively.
This recipe for the gradient reconstruction is justified by the commuting property in the following proposition.
\begin{proposition}[Commuting property]
  For all $v\in W^{1,1}(T)$, it holds that
\begin{equation}\label{eq:commut.GT}
  \GT\IT v = \lproj[T]{k} (\GRAD v).
\end{equation}
\end{proposition}
\begin{proof}
  Plugging the definition~\eqref{eq:IT} of $\IT$ into~\eqref{def:GT.bis}, it is inferred for all $\bphi\in\Poly{k}(T)^d$ that
  $$
  \begin{aligned}
    (\GT\IT v,\bphi)_T
    &= -(\lproj[T]{k} v,\DIV\bphi)_T
    + \sum_{F\in\Fh[T]} (\lproj[F]{k} v, \bphi\SCAL\normal_{TF})_F
    \\
    &= -(v,\DIV\bphi)_T
    + \sum_{F\in\Fh[T]} (v, \bphi\SCAL\normal_{TF})_F,
  \end{aligned}
  $$
  where, to cancel the projectors in the second line, we have used~\eqref{def:lproj}
  together with the fact that $\DIV\bphi\in\Poly{k-1}(T)\subset\Poly{k}(T)$ and that $\restrto{\bphi}{F}\SCAL\normal_{TF}\in\Poly{k}(F)$ for all $F\in\Fh[T]$.
  Integrating by parts the right-hand side, we conclude that
  $$
  \int_T(\GT\IT v - \GRAD v)\SCAL\bphi = 0\qquad\forall\bphi\in\Poly{k}(T)^d.
  $$
  Comparing with~\eqref{def:lproj} the conclusion follows.
\end{proof}

For further use, we note the following formula inferred from \eqref{def:GT.bis} integrating by parts the first term in the right-hand side:
For all $\sv\in\UT$ and all $\bphi\in\Poly{k}(T)^d$,
\begin{equation}
  \label{def:GT}
  (\GT \sv,\bphi)_T
  = (\GRAD \unv[T],\bphi)_T
  + \sum_{F\in\Fh[T]} (\unv[F]-\unv[T],\bphi\SCAL\normal_{TF})_F.
\end{equation}
We also define the local potential reconstruction operator $\pT:\UT\to\Poly{k+1}(T)$ such that, for all $\sv\in\UT$, 
\begin{equation}\label{eq:pT}
	\begin{aligned}
  &\int_T(\GRAD\pT\sv-\GT\sv)\SCAL\GRAD w=0\mbox{ for all $w\in\Poly{k+1}(T)$,}\\
	&\int_T (\pT\sv-\unv[T]) = 0.
	\end{aligned}
\end{equation}
As already noticed in Ref.~\cite{Di-Pietro.Ern.ea:14} (cf., in particular, Eq. (17) therein), we have the following relation which establishes a link between the potential reconstruction $\pT$ composed with the interpolation operator $\IT$ defined by \eqref{eq:IT} and the elliptic projector $\eproj[T]{k+1}$ defined by \eqref{eq:eproj}:
\begin{equation}\label{eq:pT.eproj}
  \pT\circ\IT=\eproj[T]{k+1}.
\end{equation}

The local gradient and potential reconstructions give rise to the global gradient operator $\Gh:\Uh\to\Poly{k}(\Th)^d$ and potential reconstruction $\ph:\Uh\to\Poly{k+1}(\Th)$ such that, for all $\sv[h]\in\Uh$,
\begin{equation}\label{def:Gh.ph}
  \text{$\restrto{(\Gh\sv[h])}{T}=\GT\sv[T]$ and $\restrto{(\ph\sv[h])}{T}=\pT\sv[T]$ for all $T\in\Th$.}
\end{equation}

\subsubsection{Discrete problem}

For all $T\in\Th$, we define the local function $\ascT:\UT\times\UT\to\Real$ such that
\begin{subequations}\label{def:hho.scheme}
  \begin{equation}
    \label{eq:hho-loc}
    \ascT(\su,\sv)\eqbydef{}\int_T \bfa(\vec{x},\GT\su(\vec{x}))\cdot\GT\sv(\vec{x})\dx + s_T(\su,\sv),
  \end{equation}
  with $s_T:\UT\times\UT\to\Real$ the stabilisation term such that
  \begin{equation}\label{eq:hho-stab}
    s_T(\su,\sv)\eqbydef
    \sum_{F\in\Fh[T]}h_F^{1-p}\int_F \left|\dTF\su\right|^{p-2}\dTF\su~\dTF\sv.
  \end{equation}
  In \eqref{eq:hho-stab}, the scaling factor $h_F^{1-p}$ ensures the dimensional homogeneity of the terms composing $A_T$, and the face-based residual operator $\dTF:\UT\to\Poly{k}(F)$ is defined such that, for all $\sv\in\UT$,
  \begin{equation}\label{eq:dTF}
    \dTF\sv\eqbydef\lproj[F]{k}(\unv[F] - \restrto{(\pT\sv)}{F}) - \restrto{(\lproj[T]{k}(\unv[T] - \pT\sv))}{F}.
  \end{equation}
  A global function $\asch:\Uh\times\Uh\to\Real$ is assembled element-wise from local contributions setting
  \begin{equation}
    \label{eq:hho-assembly}
    \asch(\su[h],\sv[h])\eqbydef\sum_{T\in\Th} \ascT(\su[T],\sv[T]).
  \end{equation}
  Boundary conditions are strongly enforced by considering the following subspace of $\Uh$:
  \begin{equation}\label{def:UhD}
    \UhD\eqbydef\left\{
    \sv[h]\in\Uh\st \unv[F]\equiv 0\quad\forall F\in\Fhb
    \right\}.
  \end{equation}
  The HHO approximation of problem \eqref{pde:plap.weak} reads:
  \begin{equation}
    \label{eq:hho-glob}
    \text{Find $\su[h]\in \UhD$ such that, for all $\sv[h]\in \UhD$, 
      $\asch(\su[h],\sv[h]) = \int_{\Omega} f \unv[h]$.}
  \end{equation}
\end{subequations}
For a discussion on the existence and uniqueness of a solution to \eqref{def:hho.scheme} we refer the reader to Theorem 4.5 and Remark 4.7 in Ref.~\cite{Di-Pietro.Droniou:16}.

\subsection{Error estimates}\label{sec:error.est}

We state in this section an error estimate in terms of the following discrete $W^{1,p}$-seminorm on $\Uh$:
\begin{equation}\label{def:norm.1}
  \norm[1,p,h]{\sv[h]}\eqbydef\left(\sum_{T\in\Th} \norm[1,p,T]{\sv}^p\right)^{\frac1p},
\end{equation}
where, for all $T\in\Th$,
$$
  \norm[1,p,T]{\sv}\eqbydef\left(
  \norm[L^p(T)^{d}]{\GRAD \pT\sv[T]}^p
  + s_T(\sv[T],\sv[T])  \right)^{\frac1p}.
$$
\begin{proposition}[{Norm $\norm[1,p,h]{{\cdot}}$}]
  The map $\norm[1,p,h]{{\cdot}}$ defines a norm on $\UhD$.
\end{proposition}
\begin{proof}
  The semi-norm property is trivial, so it suffices to prove that, for all $\sv[h]\in\UhD$, $\norm[1,p,h]{\sv[h]}=0\implies\sv[h]=\underline{\mathsf{0}}$.
  Let $\sv[h]\in\UhD$ be such that $\norm[1,p,h]{\sv[h]}=0$.
  The semi-norm equivalence proved in Lemma~5.2 of Ref.~\cite{Di-Pietro.Droniou:16} (see also
  \eqref{eq:equiv.snorms} below) shows that
  \[
  \sum_{T\in\Th}\norm[L^p(T)^d]{\GRAD \unv[T]}^p+\sum_{T\in\Th}\sum_{F\in\Fh[T]}
  h_F^{1-p}\norm[L^p(F)]{\unv[F]-\restrto{(\unv[T])}{F}}^p=0.
  \]
  Hence, all $(\unv[T])_{T\in\Th}$ are constant polynomials and, for any $F\in\Fh[T]$,
  $\unv[F]\equiv \restrto{(\unv[T])}{F}$.
  Starting from boundary mesh elements $T\in\Th$ for which there exists $F\in\Fh[T]\cap\Fhb$, using the fact that $\unv[F]\equiv 0$ whenever $F\in\Fhb$, and proceeding from neighbour to neighbour towards the interior of the domain, we infer
  that $\unv[T]\equiv0$ for all $T\in\Th$ and $\unv[F]\equiv 0$ for all $F\in\Fh$.
\end{proof}

The regularity assumptions on the exact solution are expressed in terms of the broken $W^{s,p}$-spaces defined by
\[
W^{s,p}(\Th)\eqbydef\{v\in L^p(\Omega)\,:\,\forall T\in\Th\,,\;v\in W^{s,p}(T)\},
\]
which we endow with the norm
\[
\norm[W^{s,p}(\Th)]{v}\eqbydef\left(\sum_{T\in\Th}\norm[W^{s,p}(T)]{v}^p\right)^{\frac1p}.
\]
Notice that, if $v\in W^{s,p}(\Th)$ for $h=h_1$ and $h=h_2$ with $h_1,h_2\in{\cal H}$, then $\norm[W^{s,p}(\mathcal{T}_{h_1})]{v}=\norm[W^{s,p}(\mathcal{T}_{h_2})]{v}$.
Our main result is summarised in the following theorem, whose proof makes use of the approximation results for the elliptic projector stated in Theorems \ref{thm:Wsp.approx} and \ref{thm:Wsp.approx.trace}; cf. Remark \ref{rem:role.T1.T2} for further insight into their role.

\begin{theorem}[Error estimate]\label{thm:est.error}
  Let the assumptions in \eqref{assum:gen} hold, and let $u$ solve \eqref{pde:plap.weak}.
  Let a polynomial degree $k\ge 0$, a mesh $\Th$, and a set of faces $\Fh$ be fixed, and let $\su[h]$ solve \eqref{def:hho.scheme}.
  Assume the additional regularity $u\in W^{k+2,p}(\Th)$ and $\bfa(\cdot,\GRAD u)\in W^{k+1,p'}(\Th)^d$ (with $p'=\frac{p}{p-1}$), and define the quantity $\Eh(u)$ as follows:
  \begin{subequations}\label{def:Eh}
    \begin{asparaitem}
    \item If $p\ge 2$,
      \begin{equation}\label{def:Eh.2}
        \Eh(u)\eqbydef h^{k+1}\seminorm[W^{k+2,p}(\Th)]{u}
        +h^{\frac{k+1}{p-1}}
        \Big(\seminorm[W^{k+2,p}(\Th)]{u}^{\frac{1}{p-1}}+
        \seminorm[W^{k+1,p'}(\Th)^d]{\bfa(\cdot,\GRAD u)}^{\frac{1}{p-1}}
        \Big);
      \end{equation}
    \item If $p<2$,
      \begin{equation}\label{def:Eh.1}
        \Eh(u)\eqbydef h^{(k+1)(p-1)}\seminorm[W^{k+2,p}(\Th)]{u}^{p-1}+h^{k+1}\seminorm[W^{k+1,p'}(\Th)^d]{\bfa(\cdot,\GRAD u)}.
      \end{equation}
    \end{asparaitem}
  \end{subequations}
  Then, there exists a real number $C>0$ depending only on $\Omega$, $k$, the mesh regularity parameter $\varrho$ defined in Assumption \ref{def:adm.Th}, the coefficients $p$, $\upa$, $\coera$, $\lipa$, $\mona$ defined in \eqref{assum:gen}, and an upper bound of $\norm[L^{p'}(\Omega)]{f}$ such that
  \begin{equation}\label{eq:est.error}
    \norm[1,p,h]{\Ih u-\su[h]}\le C\Eh(u).
  \end{equation}
  \label{th:error.est}
\end{theorem}

\begin{proof}
  See Section \ref{sec:appl:proofs:error.est}.
\end{proof}

Some remarks are of order.
\begin{remark}[Orders of convergence]\label{rem:order.of.conv}
  From \eqref{eq:est.error}, it is inferred that the approximation error in the discrete $W^{1,p}$-norm scales as the dominant terms in $\Eh$, namely
  \begin{alignat}{2}
    \label{order1}
      &h^{\frac{k+1}{p-1}} &\qquad& \text{if $p\ge 2$},
      \\
      \label{order2}
      &h^{(k+1)(p-1)} &\qquad& \text{if $1<p<2$.}
    \end{alignat}
    Let us discuss how these orders compare with some known results
    for $\Poly{1}$ approximations of the $p$-Laplacian, starting from conforming approximations.
    In Theorem 5.3.5 of Ref.~\cite{Ciarlet:91}, an order $h^{\frac1{p-1}}$ is established in
    the case $p\ge 2$, which is identical to \eqref{order1} with $k=0$.
    The case $1<p<2$ is considered in Ref.~\cite{Glowinski.Marrocco:75}, and
    an estimate in $h^{\frac1{3-p}}$ is proved. This order
    is better than \eqref{order2} for $k=0$, but the proof relies
    on the fact that the $\Poly{1}$ finite element method is a \emph{conforming}
    method (see Eq. (6.5) in this reference). These latter
    rates are improved to order $h$ in Ref.~\cite{Barrett.Liu:94},
    but under a $C^{2,\alpha}$ regularity assumption on $u$.
    The case $p\ge 2$ is also considered in Ref.~\cite{Barrett.Liu:94},
    and an order $h$ estimate is obtained in $W^{1,q}$ for some
    $q<2$ (which is weaker than \eqref{eq:est.error}),
    under a strictly positive lower bound on $|f|$ (and, thus, on $|\nabla u|$).
    All these analyses strongly use the conformity of the $\Poly{1}$ element.
    
    Let us now consider nonconforming approximations, to which the HHO method proposed in this work belongs.
    The rates of convergence established in Ref.~\cite{and-07-dis} for the DDFV
    method are identical to \eqref{order1}--\eqref{order2} for $k=0$.
    Similar considerations apply to the Crouzeix--Raviart approximation considered in Ref.~\cite{Liu.Yan:01}, which is strongly linked to our HHO methods for $k=0$ on matching simplicial meshes\cite{Boffi.Di-Pietro:16}.

    The numerical tests of Section~\ref{sec:num.ex} seem to indicate that our estimates are sharp at least for the case $1<p\le 2$.
    This suggests, in turn, that the order of convergence depends on the polynomial degree $k$, on the index $p$, and on the conformity properties of the method.
    We emphasize that, despite the reduced order of convergence with respect to the best estimate for conforming $\Poly{1}$ schemes and $1<p<2$, the HHO method proposed here has the key advantage of supporting general meshes, as well as arbitrary orders of approximation.
    
\end{remark}

\begin{remark}[Role of the various terms] There is a nice parallel between the various error terms
in \eqref{def:Eh} and the error estimate obtained for the gradient discretisation method in Ref.~\cite{gdm}.
In the framework of the gradient discretisation method\cite{Eymard.Guichard.ea:12,Droniou.Eymard.ea:13}, the accuracy of a scheme is essentially assessed through two quantities: a measure $W_{\mathcal D}$ of the default of conformity of the scheme,
and a measure $S_{\mathcal D}$ of the consistency of the scheme.
In \eqref{def:Eh}, the terms involving $\seminorm[W^{k+1,p'}(\Th)^d]{\bfa(\cdot,\GRAD u)}$ estimate the contribution to the error of the default of conformity of the method, and the terms involving $\seminorm[W^{k+2,p}(\Th)]{u}$
come from the consistency error of the method.
\end{remark}

From the convergence result in Theorem \ref{thm:est.error}, we can infer an error estimate on the potential reconstruction $\ph\su[h]$ and on its jumps measured through the stabilisation function $s_T$.
\begin{corollary}[Convergence of the potential reconstruction]\label{cor:error.int}
  Under the notations and assumptions in Theorem \ref{th:error.est}, and denoting by $\GRADh$ the broken gradient on $\Th$, we have
  \begin{multline}\label{eq:est.error:pT}
    \left(
    \norm[L^p(\Omega)^d]{\GRADh (u-\ph\su[h])}^p
    +\sum_{T\in\Th}s_T(\su,\su)
    \right)^{\frac1p}\\
    \le C \left(\Eh(u)  
    + h^{k+1}\seminorm[W^{k+2,p}(\Th)]{u}
    \right),
  \end{multline}
where $C$ has the same dependencies as in Theorem \ref{th:error.est}.
\end{corollary}
\begin{proof}
  See Section \ref{sec:appl:proofs:error.int}.
\end{proof}%

  \begin{remark}[Variations]
    Following Remark 4.4 in Ref.~\cite{Di-Pietro.Droniou:16}, variations of the HHO scheme \eqref{def:hho.scheme} are obtained replacing the space $\UT$ defined by \eqref{eq:UT} by
    $$
    \UT[l,k]\eqbydef
    \Poly{l}(T)\times\left(
    \bigtimes_{F\in\Fh}\Poly{k}(F)
    \right),
    $$
    for $k\ge 0$ and $l\in\{k-1,k,k+1\}$.
    For the sake of simplicity, we consider the case $l=k-1$ only when $k\ge 1$ (technical modifications, not detailed here, are required for $k=0$ and $l=k-1$ owing to the absence of element DOFs).
    The interpolant $\IT$ naturally has to be replaced with $\IT[l,k]v\eqbydef (\lproj[T]{l}v, (\lproj[F]{k}v)_{F\in\Fh[T]})$.
    The definitions \eqref{def:GT.bis} of $\GT$ and \eqref{eq:pT} of $\pT$ remain formally the same (only the domain of the operators changes), and a close inspection shows that both key properties~\eqref{eq:commut.GT} and~\eqref{eq:pT.eproj} remain valid for all the proposed choices for $l$ (replacing, of course, $\IT$ with $\IT[l,k]$ in both cases).
    In the expression \eqref{eq:hho-stab} of the penalization bilinear form $s_T$, we replace the face-based residual $\dTF$ defined by \eqref{eq:dTF} with a new operator $\dTF[l,k]:\UT[l,k]\to\Poly{k}(F)$ such that, for all $\sv\in\UT[l,k]$,
    $
    \dTF[l,k]\sv\eqbydef\lproj[F]{k}[\unv[F]-
    \restrto{(\pT\sv)}{F} - 
    \restrto{(\lproj[T]{l}(\unv[T]-\pT\sv))}{F} ].
    $
    Up to minor modifications, the proof of Theorem \ref{thm:est.error} remains valid, and therefore so is the case for the error estimates \eqref{eq:est.error} and \eqref{eq:est.error:pT}.
  \end{remark}%

\subsection{Proof of the error estimates}\label{sec:appl:proofs}

In this section, we write $A\lesssim B$ for $A\le MB$ with $M$ having the same dependencies
as $C$ in Theorem \ref{th:error.est}. The notation $A\approx B$ means $A\lesssim B$ and $B\lesssim A$.

\subsubsection{Proof of Theorem \ref{th:error.est}}\label{sec:appl:proofs:error.est}
  The proof is split into several steps.
  In {\bf Step 1} we obtain an initial estimate involving, on the left-hand side, $\bfa$ and $s_T$, and, on the right-hand side, a sum of four terms.
  In {\bf Step 2} we prove that the left-hand side of this estimate provides an upper bound of the approximation error $\norm[1,p,h]{\Ih u-\su[h]}$.
  Then, in {\bf Steps 3--5}, we estimate each of the four terms in the right-hand side of the original estimate.
  Combined with the result of {\bf Step 2}, these estimates prove \eqref{eq:est.error}.

  Throughout the proof, to alleviate the notation, we write $\mathcal O(X)$ for a quantity that satisfies $|\mathcal O(X)|\lesssim X$, and we abridge $\Ih u$ into $\shu[h]$.

  We will need the following equivalence of local seminorms, established in Lemma 5.2 of Ref.~\cite{Di-Pietro.Droniou:16}:
  For all $\sv\in \UT$,
  \begin{equation}\label{eq:equiv.snorms}
    \begin{aligned}
    \norm[1,p,T]{\sv}
      &\approx \bigg(\norm[L^p(T)^d]{\GRAD \unv[T]}^p+ \hspace{-1ex}\sum_{F\in\Fh[T]} h_F^{1-p}\norm[L^p(F)]{\unv[F]-\unv[T]}^p\bigg)^{\frac1p} \\
      &\approx \bigg(\norm[L^p(T)^d]{\GT\sv[T]}^p + s_T(\sv[T],\sv[T])\bigg)^{\frac1p}.
    \end{aligned}
  \end{equation}

  \begin{asparaenum}[\bf Step 1.]
    
  \item\emph{Initial estimate.}
    Let $\sv[h]$ be a generic element of $\UhD$, and denote by $\sv\in\UT$ its restriction to a generic mesh element $T\in\Th$.
    In this step, we estimate the error made when using $\shu[h]$, instead of $\su[h]$, in the scheme, namely
    \begin{equation}\label{eq:cEh}
			\begin{aligned}
      \cE(\sv[h])
      \eqbydef{}&\sum_{T\in\Th} \int_T \left[\bfa(\vec{x},\GT\shu[T])-\bfa(\vec{x},\GT\su[T])\right]\SCAL\GT\sv[T]\\
      {}&+\sum_{T\in\Th} (s_T(\shu[T],\sv[T])-s_T(\su[T],\sv[T])).
			\end{aligned}
    \end{equation}

    Let $T\in\Th$ be fixed. Setting
    \begin{equation}\label{eq:def.T1T}
      \term_{1,T}\eqbydef\norm[L^{p'}(T)^d]{\bfa(\cdot,\GT\shu[T])-\bfa(\cdot,\GRAD u)},
    \end{equation}
    by the H\"older inequality we infer
    \begin{multline*}
    \int_T \bfa(\vec{x},\GT\shu[T](\vec{x}))\SCAL\GT\sv[T](\vec{x})\dx
      \\
    =\int_T \bfa(\vec{x},\GRAD u(\vec{x}))\SCAL \GT\sv[T](\vec{x})\dx + \mathcal O(\term_{1,T})\norm[L^p(T)^d]{\GT\sv[T]}.
    \end{multline*}
    To benefit from the definition \eqref{def:GT.bis} of $\GT\sv[T]$, we approximate $\bfa(\cdot,\GRAD u)$ by its $L^2$-orthogonal projetion on the polynomial space $\Poly{k}(T)^d$.
    We therefore introduce
    \begin{equation}\label{eq:def.T2T}
      \term_{2,T}\eqbydef\norm[L^{p'}(T)^d]{\bfa(\cdot,\GRAD u)-\lproj[T]{k}\bfa(\cdot,\GRAD u)},
    \end{equation}
    and we have
    \begin{multline}\label{est.error:1}
      \int_T \bfa(\vec{x},\GT\shu[T](\vec{x}))\SCAL\GT\sv[T](\vec{x})\dx
      =
      \\
      \int_T \lproj[T]{k}\bfa(\vec{x},\GRAD u(\vec{x}))\SCAL \GT\sv[T](\vec{x})\dx
      +\mathcal O(\term_{1,T}+\term_{2,T})\norm[L^p(T)^d]{\GT\sv[T]}.
    \end{multline}
    Using \eqref{def:GT} with $\bphi = \lproj[T]{k}\bfa(\cdot,\GRAD u)$, the first term in the right-hand side rewrites
    \begin{multline*}
    \int_T \lproj[T]{k}\bfa(\vec{x},\GRAD u(\vec{x}))\SCAL \GT\sv[T](\vec{x})\dx
      \\
    = (\lproj[T]{k}\bfa(\cdot,\GRAD u),\GRAD \unv[T])_T + \sum_{F\in\Fh[T]}(\lproj[T]{k}\bfa(\cdot,\GRAD u)\SCAL	
    \normal_{TF},\unv[F]-\unv[T])_F.
    \end{multline*}
    We now want to eliminate the projectors $\lproj[T]{k}$, in order to utilise the fact that $u$ is a solution
    to \eqref{pde:plap}.
    In the first term, the projector $\lproj[T]{k}$ can be cancelled simply by observing that $\GRAD\unv[T]\in\Poly{k-1}(T)^d\subset\Poly{k}(T)^d$, whereas for the second term we introduce an error that we want to control by
    \begin{equation}\label{def:T3T}
      \term_{3,T}\eqbydef\left(\sum_{F\in\Fh[T]}h_F\norm[L^{p'}(F)]{\bfa(\cdot,\GRAD u)-\lproj[T]{k}\bfa(\cdot,\GRAD u)}^{p'}\right)^{\frac{1}{p'}}
    \end{equation}
    (this quantity is well defined since $\bfa(\cdot,\GRAD u)\in W^{1,p'}(T)^d$ by assumption).
    We therefore have, using the H\"older inequality,
    \begin{align*}
      \int_T \lproj[T]{k}\bfa(\vec{x},\GRAD u(\vec{x}))\SCAL \GT\sv[T](\vec{x})\dx
      ={}&
      (\bfa(\cdot,\GRAD u),\GRAD \unv[T])_T\\
      {}&+\sum_{F\in\Fh[T]}(\bfa(\cdot,\GRAD u)\SCAL\normal_{TF},\unv[F]-\unv[T])_F
      \\
      &+ \mathcal O(\term_{3,T})\left(\sum_{F\in\Fh[T]}h_F^{1-p}\norm[L^p(F)]{\unv[F]-\unv[T]}^p \right)^{\frac{1}{p}}.
    \end{align*}
    We plug this expression into \eqref{est.error:1} and use the equivalence of seminorms \eqref{eq:equiv.snorms} to obtain
    \begin{align*}
      \int_T \bfa(\vec{x},\GT\shu[T](\vec{x}))\SCAL\GT\sv[T](\vec{x})\dx
      ={}& (\bfa(\cdot,\GRAD u),\GRAD \unv[T])_T +\hspace{-1.5ex} \sum_{F\in\Fh[T]}(\bfa(\cdot,\GRAD u)\SCAL
      \normal_{TF},\unv[F]-\unv[T])_F\\
      &+\mathcal O(\term_{1,T}+\term_{2,T}+\term_{3,T})\norm[1,p,T]{\sv[T]}.
    \end{align*}
    Integrating by parts the first term in the right-hand side and writing $-\div(\bfa(\cdot,\GRAD u))=f$ in $T$, we arrive at
    \begin{multline*}
      \int_T \bfa(\vec{x},\GT\shu[T](\vec{x}))\SCAL\GT\sv[T](\vec{x})\dx
      =\\
      (f,\unv[T])_T + \sum_{F\in\Fh[T]}(\bfa(\cdot,\GRAD u)\SCAL \normal_{TF},\unv[F])_F
      +\mathcal O(\term_{1,T}+\term_{2,T}+\term_{3,T})\norm[1,p,T]{\sv[T]}.
    \end{multline*}
    We then sum over $T\in\Th$, use $\bfa(\cdot,\GRAD u)\SCAL\normal_{T_1F}=-\bfa(\cdot,\GRAD u)\SCAL\normal_{T_2F}$ on every interface $F\in\Fhi$ such that $F\in \Fh[T_1]\cap \Fh[T_2]$ for distinct mesh elements $T_1,T_2\in\Th$ (this is because $-\div(\bfa(\cdot,\GRAD u))\in L^{p'}(\Omega)$) together with $\unv[F]=0$ for every boundary face $F\in\Fhb$ to infer
    $$
    \sum_{T\in\Th}\sum_{F\in\Fh[T]}(\bfa(\cdot,\GRAD u)\SCAL\normal_{TF},\unv[F])_F = 0,
    $$
    invoke the scheme \eqref{def:hho.scheme}, and use the H\"older inequality on the $\mathcal O$ terms to write
    \begin{multline*}
    \sum_{T\in\Th} \int_T \left[\bfa(\vec{x},\GT\shu[T](\vec{x}))-\bfa(\vec{x},\GT\su[T](\vec{x}))\right]\SCAL\GT\sv[T](\vec{x})\dx
    -\sum_{T\in\Th} s_T(\su[T],\sv[T])
    \\
    = \mathcal O(\term_{1}+\term_{2}+\term_{3})\norm[1,p,h]{\sv[h]}
    \end{multline*}
    where, for $i\in\{1,2,3\}$, we have set
    \begin{equation}\label{eq:term.i}
      \term_{i}\eqbydef\left(\sum_{T\in\Th}\term_{i,T}^{p'}\right)^{\frac{1}{p'}}.
    \end{equation}
    Finally, introducing the last error term
    \begin{equation}\label{eq:def.T4}
      \term_4\eqbydef\sup_{\sv[h]\in\Uh,\sv[h]\neq\underline{\mathsf{0}}_h} \frac{\sum_{T\in\Th}s_T(\shu[T],\sv[T])}{\norm[1,p,h]{\sv[h]}},
    \end{equation}
    we have
    \begin{equation}
      \cE(\sv[h])
      = \mathcal O(\term_{1}+\term_{2}+\term_{3}+\term_4)\norm[1,p,h]{\sv[h]}.
      \label{initial.est}
    \end{equation}

    \medskip

  \item\emph{Lower bound for $\cE(\shu[h]-\su[h])$.}

    Let, for the sake of conciseness, $\se[h]\eqbydef\shu[h]-\su[h]$.
    The goal of this step is to find a lower bound for $\cE(\se[h])$ in terms of the error measure $\norm[1,p,h]{\se[h]}$.
    To this end, we let $\sv[h]=\se[h]$ in the definition \eqref{eq:cEh} of $\cE$ and distinguish two cases.

    \underline{Case $p\ge 2$}: Using for all $T\in\Th$ the bound \eqref{amon:p2p} below with $\vec{\xi}=\GT\shu[T]$ and $\vec{\eta}=\GT\su$ for the first term in the right-hand side of \eqref{eq:cEh}, the definition \eqref{eq:hho-stab} of $s_T$ and, for all $F\in\Fh[T]$, the bound \eqref{eq:mon.sT.2} below with $t=\dTF\shu[T]$ and $r=\dTF\su$ for the second, and concluding by the norm equivalence \eqref{eq:equiv.snorms}, we have
    \begin{equation}
      \cE(\se[h])\gtrsim{} \sum_{T\in\Th} \left(\norm[L^p(T)^d]{\GT\se[T]}^p+
      \sum_{F\in\Fh[T]} h_F^{1-p}\norm[L^p(F)]{\dTF\se[T]}^p\right)\gtrsim{}\norm[1,p,h]{\se[h]}^p.
      \label{est:Ah.1}
    \end{equation}

    \smallskip

    \underline{Case $p<2$}: Let an element $T\in\Th$ be fixed.
    Applying \eqref{amon:p2m} below to $\vec{\xi}=\GT\shu[T]$ and $\vec{\eta}=\GT\su[T]$, integrating over $T$ and using the H\"older inequality with exponents $\frac{2}{p}$ and $\frac{2}{2-p}$, we get
    \begin{multline*}
    \norm[L^p(T)^d]{\GT\se[T]}^p
    \lesssim
    \left(
    \int_T [\bfa(\vec{x},\GT\shu[T](\vec{x}))-\bfa(\vec{x},\GT\su[T](\vec{x}))]\SCAL\GT\se[T](\vec{x})\dx
    \right)^{\frac{p}{2}}\hspace{-1ex}
    \\
    \times \left(\norm[L^p(T)^d]{\GT\shu[T]}^p
    + \norm[L^p(T)^d]{\GT\su[T]}^p\right)^{\frac{2-p}{2}}.
    \end{multline*}

    Summing over $T\in\Th$ and using the discrete H\"older inequality, we obtain
    \begin{equation}\label{est:Ah.2}
      \norm[L^p(\Omega)^d]{\Gh\se[h]}^p
      \lesssim\cE(\se[h])^{\frac{p}{2}}
      \times
      \left(\norm[L^p(\Omega)^d]{\Gh\shu[h]}^p+\norm[L^p(\Omega)^d]{\Gh\su[h]}^p\right)^{\frac{2-p}{2}}.
    \end{equation}
    A similar reasoning starting from \eqref{eq:mon.sT.1} with $t=h_F^{\frac{1-p}{p}}\dTF\shu[T]$ and $r=h_F^{\frac{1-p}{p}}\dTF\su[T]$, integrating over $F$, summing over $F\in\Fh[T]$ and using the H\"older inequality gives
    \[
    s_T(\se[T],\se[T])\lesssim{} \left(s_T(\shu[T],\se[T])-s_T(\su[T],\se[T])\right)^{\frac{p}{2}}
    \left(s_T(\shu[T],\shu[T])+s_T(\su[T],\su[T])\right)^{\frac{2-p}{2}}.
    \]
    Summing over $T\in\Th$ and using the discrete H\"older inequality, we get
    \begin{equation}\label{est:Ah.3}
      \sum_{T\in\Th}s_T(\se[T],\se[T])\lesssim
      \cE(\se[h])^{\frac{p}{2}}\times
      \left(\sum_{T\in\Th}s_T(\shu[T],\shu[T]) + \sum_{T\in\Th}s_T(\su[T],\su[T])\right)^{\frac{2-p}{2}}.
    \end{equation}
    Combining \eqref{est:Ah.2} and \eqref{est:Ah.3}, and using the seminorm equivalence \eqref{eq:equiv.snorms} leads to
    \[
    \norm[1,p,h]{\se[h]}^p\lesssim \cE(\se[h])^{\frac{p}{2}}\times
    \left(\norm[1,p,h]{\shu[h]}^p+\norm[1,p,h]{\su[h]}^p\right)^{\frac{2-p}{2}}.
    \]
    From the $W^{1,p}$-boundedness of $\IT$ and the a priori bound on $\norm[1,p,h]{\su[h]}$ proved in 
    Propositions 7.1 and 6.1 of Ref.~\cite{Di-Pietro.Droniou:16}, respectively, we infer that
    \begin{equation}\label{eq:stab.IT}
      \norm[1,p,h]{\shu[h]}\lesssim\norm[W^{1,p}(\Omega)]{u}\lesssim 1
      \text{ and }
      \norm[1,p,h]{\su[h]}\lesssim\norm[L^{p'}(\Omega)]{f}^{\frac{1}{(p-1)}}\lesssim 1,
    \end{equation}
    so that
    \begin{equation}\label{est:Ah.4}
      \norm[1,p,h]{\se[h]}^2\lesssim \cE(\se[h]).
    \end{equation}

    In conclusion, combining the initial estimate \eqref{initial.est} with $\sv[h]=\se[h]$ with the bounds \eqref{est:Ah.1} (if $p\ge 2$) and \eqref{est:Ah.4} (if $p<2$), we obtain
    \begin{equation}\label{initial.est.2}
      \begin{aligned}
        &\mbox{If $p\ge 2$}:\quad
        \norm[1,p,h]{\se[h]} \lesssim \mathcal O\left(\term_{1}^{\frac{1}{p-1}}+\term_{2}^{\frac{1}{p-1}}+\term_{3}^{\frac{1}{p-1}}+\term_4^{\frac{1}{p-1}}\right)\,,\\
        &\mbox{If $p<2$}:\quad
        \norm[1,p,h]{\se[h]} \lesssim \mathcal O\left(\term_{1}+\term_{2}+\term_{3}+\term_4\right).
      \end{aligned}	
    \end{equation}

    \medskip

  \item\emph{Estimate of $\term_1$.}

    Recall that, by \eqref{eq:term.i} and \eqref{eq:def.T1T},
    \[
    \term_1=\left(\sum_{T\in\Th}\norm[L^{p'}(T)^d]{\bfa(\cdot,\GT\shu[T])-\bfa(\cdot,\GRAD u)}^{p'}\right)^{\frac{1}{p'}}.
    \]
    Notice also that, by \eqref{eq:commut.GT}, $\GT\shu[T]=\GT\IT u=\lproj[T]{k}(\GRAD u)$.
    Thus, using the approximation properties of $\lproj[T]{k}$ summarised in Lemma \ref{lem:Wkp.interp} below (with $v=\partial_i u$ for $i=1,\ldots,d$), we infer
    \begin{equation}\label{eq:optim.GT}
      \norm[L^p(T)^d]{\GT\shu[T]-\GRAD u}\lesssim h_T^{k+1}\seminorm[W^{k+2,p}(T)]{u}.
    \end{equation}

    \underline{Case $p\ge 2$}: Assume first $p>2$. 
    Recalling \eqref{hyp:alip}, and using the generalised H\"older inequality with exponents $(p',p,r)$ such that $\frac{1}{p'}=\frac{1}{p}+\frac{1}{r}$ (that is $r=\frac{p}{p-2}$) together with \eqref{eq:optim.GT} yields, for all $T\in\Th$,
    \begin{align*}
      &\norm[L^{p'}(T)^d]{\bfa(\cdot,\GT\shu[T])-\bfa(\cdot,\GRAD u)}
      \\
      &\qquad\lesssim \norm[L^p(T)^d]{\GT\shu[T]-\GRAD u}\left(\norm[L^{p}(T)^d]{\GT\shu[T]}^{p-2}+
      \norm[L^{p}(T)^d]{\GRAD u}^{p-2}\right)\\
      &\qquad\lesssim h_T^{k+1}\seminorm[W^{k+2,p}(T)]{u}\left(\norm[L^{p}(T)^d]{\GT\shu[T]}^{p-2}+
      \norm[L^{p}(T)^d]{\GRAD u}^{p-2}\right).
    \end{align*}
    This relation is obviously also valid if $p=2$. 
    We then sum over $T\in\Th$ and use, as before, the generalised H\"older inequality, 
    the estimate $\norm[L^p(\Omega)^d]{\Gh\shu[h]}\lesssim \norm[1,h,p]{\shu[h]}
    \lesssim \norm[W^{1,p}(\Omega)]{u}$
    (see \eqref{eq:equiv.snorms} and Proposition 7.1 in Ref.~\cite{Di-Pietro.Droniou:16}),
    and \eqref{eq:stab.IT} to infer
    \[
    \term_1 \lesssim h^{k+1}\seminorm[W^{k+2,p}(\Th)]{u}\left(\norm[L^p(\Omega)^d]{\Gh\shu[h]}^{p-2}
    +\norm[W^{1,p}(\Omega)]{u}^{p-2}\right)\lesssim h^{k+1}\seminorm[W^{k+2,p}(\Th)]{u}.
    \]

    \smallskip

    \underline{Case $p<2$}: By \eqref{alip:p2} below, $\norm[L^{p'}(T)^d]{\bfa(\cdot,\GT\shu[T])-\bfa(\cdot,\GRAD u)}\lesssim
    \norm[L^{p}(T)^d]{\GT\shu[T]-\GRAD u}^{p-1}$.
    Use then \eqref{eq:optim.GT} and sum over $T\in\Th$ to obtain
    $\term_1 \lesssim h^{(k+1)(p-1)}\seminorm[W^{k+2,p}(\Th)]{u}^{p-1}$.

    \smallskip

    In conclusion, we obtain the following estimates on $\term_1$:
    \begin{equation}\label{est:term1}
      \begin{aligned}
        &\text{If $p\ge 2$}:\quad	\term_1 \lesssim h^{k+1}\seminorm[W^{k+2,p}(\Th)]{u},\\
        &\text{If $p<2$}:\quad	\term_1 \lesssim h^{(k+1)(p-1)}\seminorm[W^{k+2,p}(\Th)]{u}^{p-1}.
      \end{aligned}
    \end{equation}

    \medskip

  \item\emph{Estimate of $\term_2+\term_3$.}
    {}Owing to \eqref{eq:term.i} together with the definitions \eqref{eq:def.T2T} and \eqref{def:T3T} of $\term_{2,T}$ and $\term_{3,T}$, we have
    \begin{multline*}
    \term_2^{p'}+\term_3^{p'}=
    \hspace{-1ex}\sum_{T\in\Th}\Bigg(
    \norm[L^{p'}(T)^d]{\bfa(\cdot,\GRAD u)-\lproj[T]{k}(\bfa(\cdot,\GRAD u))}^{p'}
    \\
    +\sum_{F\in\Fh[T]}h_F \norm[L^{p'}(F)^d]{\bfa(\cdot,\GRAD u)-\lproj[T]{k}(\bfa(\cdot,\GRAD u))}^{p'}
    \Bigg).
    \end{multline*}
    Using the approximation properties \eqref{eq:approx.lproj.Wsp} and \eqref{eq:approx.lproj.Wsp.trace} of $\lproj[T]{k}$ with $v$ replaced by the components of $\bfa(\cdot,\GRAD u)$, $p'$ instead of $p$, and $m=0$, $s=k+1$, we get
    \[
    \term_2^{p'}+\term_3^{p'}\lesssim h^{(k+1)p'}\seminorm[W^{k+1,p'}(\Th)^d]{\bfa(\cdot,\GRAD u)}^{p'}.
    \]
    Taking the power $1/p'$ of this inequality and using $(a+b)^{\frac{1}{p'}}\le a^{\frac{1}{p'}}+b^{\frac{1}{p'}}$ leads to
    \begin{equation}\label{est:T2T3}
      \term_2+\term_3\lesssim h^{k+1}\seminorm[W^{k+1,p'}(\Th)^d]{\bfa(\cdot,\GRAD u)}.
    \end{equation}

    \medskip

  \item\emph{Estimate of $\term_4$.}

    Recall that $\term_4$ is defined by \eqref{eq:def.T4}.
    Using the H\"older inequality, we have for all $T\in\Th$,
    \[
    s_T(\shu[T],\sv[T])\lesssim s_T(\shu[T],\shu[T])^{\frac{1}{p'}} s_T(\sv[T],\sv[T])^{\frac{1}{p}}.
    \]
    Hence, using again the H\"older inequality, since $\sum_{T\in\Th}s_T(\sv[T],\sv[T])\le \norm[1,p,h]{\sv[h]}^p$ by \eqref{eq:equiv.snorms},
    \begin{equation}\label{est:T4.1}
      \term_4\lesssim \left(\sum_{T\in\Th} s_T(\shu[T],\shu[T])\right)^{\frac{1}{p'}}.
    \end{equation}
    We proceed in a similar way as in Lemma 4 of Ref.~\cite{Di-Pietro.Ern.ea:14} to estimate $s_T(\shu[T],\shu[T])$.
    Let $F\in\Fh[T]$.
    We use the definition \eqref{eq:dTF} of the face-based residual operator $\dTF$ together with the triangle inequality,
    the relation $\lproj[F]{k}\lproj[T]{k}=\lproj[T]{k}$, the $L^p(F)$-boundedness \eqref{eq:boundedness.lproj} of $\lproj[F]{k}$, the equality $\pT\shu[T]=\pT\IT u=\eproj[T]{k+1}u$ (cf.~\eqref{eq:pT.eproj}), the trace inequality \eqref{ineq.cont.trace}, and the $L^p(T)$- and $W^{1,p}(T)$-boundedness \eqref{eq:boundedness.lproj} of $\lproj[T]{k}$ to write
    \begin{equation}\label{eq:est.dTF}
      \begin{aligned}
        \norm[L^p(F)]{\dTF\shu[T]}
        &\le\norm[L^p(F)]{\lproj[F]{k}(u -\pT\shu[T])}
        + \norm[L^p(F)]{\lproj[T]{k}(u-\pT\shu[T])}
        \\
        &\lesssim\norm[L^p(F)]{u -\eproj[T]{k+1}u}
        + h_T^{-\frac1p}\norm[L^p(T)]{\lproj[T]{k}(u-\eproj[T]{k+1}u)}    
        \\
        &\qquad+ h_T^{1-\frac1p}\seminorm[W^{1,p}(T)]{\lproj[T]{k}(u-\eproj[T]{k+1}u)}      
        \\
        &\lesssim\norm[L^p(F)]{u -\eproj[T]{k+1}u}
        + h_T^{-\frac1p}\norm[L^p(T)]{u-\eproj[T]{k+1}u}\\
        &\qquad+ h_T^{1-\frac1p}\seminorm[W^{1,p}(T)]{u-\eproj[T]{k+1}u}.
      \end{aligned}
    \end{equation}
    The optimal $W^{s,p}$-estimates on the elliptic projector \eqref{eq:Wsp.approx} and \eqref{eq:Wsp.approx.trace} therefore give, for all $F\in\Fh[T]$,
    \[
    \norm[L^p(F)]{\dTF\shu[T]}
    \lesssim h_T^{k+2-\frac1p}\seminorm[W^{k+2,p}(T)]{u}.
    \]
    Raise this inequality to the power $p$, multiply by $h_F^{1-p}$, use
    $h_F^{1-p}h_T^{(k+2)p-1}\lesssim h_F^{1-p+(k+2)p-1}=h_F^{(k+1)p}\lesssim h^{(k+1)p}$,
    and sum over $F\in\Fh[T]$ to obtain
    \begin{equation}\label{eq:est.sT.shu}
      s_T(\shu[T],\shu[T])\lesssim h^{(k+1)p}\seminorm[W^{k+2,p}(T)]{u}^p.
    \end{equation}
    Substituted into \eqref{est:T4.1}, this gives
    \begin{equation}\label{est:T4}
      \term_4 \lesssim h^{(k+1)(p-1)}\seminorm[W^{k+2,p}(\Th)]{u}^{p-1}.
    \end{equation}

    \medskip

    \noindent{\bf Conclusion.} Use \eqref{est:term1}, \eqref{est:T2T3}, and \eqref{est:T4} in \eqref{initial.est.2}.
  \end{asparaenum}
\begin{remark}[Role of Theorems \ref{thm:Wsp.approx} and \ref{thm:Wsp.approx.trace}]\label{rem:role.T1.T2}
  Theorems \ref{thm:Wsp.approx} and \ref{thm:Wsp.approx.trace} are used through
Eqs. \eqref{eq:Wsp.approx} and \eqref{eq:Wsp.approx.trace} in {\bf Step 5} of the above proof to derive a bound on the stabilisation term $s_T$, when its arguments are the interpolate of the exact solution.
\end{remark}%
The following optimal approximation properties for the $L^2$-orthogonal projector were used in {\bf Step 4} of the above with $U=T\in\Th$.

\begin{lemma}[{$W^{s,p}$-approximation for $\lproj[U]{l}$}]\label{lem:Wkp.interp}
  Let $U$ be as in Theorem \ref{thm:Wsp.approx.trace}.
  Let \mbox{$s\in\{0,\ldots,l+1\}$} and $p\in [1,+\infty]$.
  Then, there exists $C$ depending only on $d$, $\varrho$, $l$, $s$ and $p$ such that, for all $v\in W^{s,p}(U)$,
  \begin{equation}\label{eq:approx.lproj.Wsp}
    \forall m\in \{0,\ldots,s\}\,:\,
    \seminorm[W^{m,p}(U)]{v-\lproj[U]{l}v}\le Ch_U^{s-m}\seminorm[W^{s,p}(U)]{v}
  \end{equation}
  and, if $s\ge 1$,
  \begin{equation}\label{eq:approx.lproj.Wsp.trace}
    \forall m\in \{0,\ldots,s-1\}\,:\,
    h_U^{\frac1p}\seminorm[{W^{m,p}(\Fh[U])}]{v-\lproj[U]{l}v}\le Ch_U^{s-m}\seminorm[W^{s,p}(U)]{v},
  \end{equation}
  with $\Fh[U]$, $W^{m,p}(\Fh[U])$ and corresponding seminorm as in Theorem \ref{thm:Wsp.approx.trace}.
\end{lemma}

\begin{proof}
  This result is a combination of Lemmas 3.4 and 3.6 from Ref.~\cite{Di-Pietro.Droniou:16}.
  We give here an alternative proof based on the abstract results of Section \ref{sec:eproj:abstract.results}.
  By Lemma \ref{lem:Lpstab.gradproj} with $\mathcal{P}=\Poly{l}(U)$, we have the following boundedness property for $\lproj[U]{l}$:
  For all $v\in L^1(U)$, $\norm[L^p(U)]{\lproj[U]{l}v}\le C \norm[L^p(U)]{v}$ with real number $C>0$ depending only on $d$, $\varrho$, and $l$. The estimate \eqref{eq:approx.lproj.Wsp} is then an immediate consequence of Lemma \ref{lem:Wsp.approx.abstract} with $q=0$ and $\proj[0,l]{U}=\lproj[U]{l}$.
  To prove \eqref{eq:approx.lproj.Wsp.trace}, proceed as in Theorem \ref{thm:Wsp.approx.trace} using \eqref{eq:approx.lproj.Wsp} in place of \eqref{eq:Wsp.approx}.
\end{proof}

\begin{corollary}[{$W^{s,p}$-boundedness of $\lproj[U]{l}$}]
  With the same notation as in Lemma~\ref{lem:Wkp.interp}, it holds, for all $v\in W^{s,p}(U)$,
  \begin{equation}\label{eq:boundedness.lproj}
    \seminorm[W^{s,p}(U)]{\lproj[U]{l}v}\le C\seminorm[W^{s,p}(U)]{v}.
  \end{equation}
\end{corollary}

\begin{proof}
  Use the triangle inequality to write $\seminorm[W^{s,p}(U)]{\lproj[U]{l}v}\le\seminorm[W^{s,p}(U)]{\lproj[U]{l}v-v}+\seminorm[W^{s,p}(U)]{v}$ and conclude using \eqref{eq:approx.lproj.Wsp} with $m=s$ for the first term.
\end{proof}
\subsubsection{Proof of Corollary \ref{cor:error.int}}\label{sec:appl:proofs:error.int}
    Let an element $T\in\Th$ be fixed and set, as in the proof of Theorem~\ref{thm:est.error}, $\shu[T]\eqbydef\IT u$.
    Recalling the definition \eqref{eq:hho-stab} of $s_T$, and using the inequality
    \begin{equation}\label{eq:bnd.(a+b)^p}
      (a+b)^p\le 2^{p-1}a^p+2^{p-1}b^p,
    \end{equation}
    it is inferred
    \begin{align}
        s_T(\su,\su)
        &= \sum_{F\in\Fh[T]}h_F^{1-p}\int_F|\dTF\su|^p
        = \sum_{F\in\Fh[T]}h_F^{1-p}\int_F|\dTF\shu[T] + \dTF(\su-\shu[T])|^p\nonumber
        \\
        &\lesssim s_T(\shu[T], \shu[T]) + s_T(\su-\shu[T],\su-\shu[T]). \label{eq:error.int:1}
    \end{align}
    On the other hand, inserting $\pT\shu[T]-\eproj[T]{k+1}u=0$ (cf. \eqref{eq:pT.eproj}), and using again \eqref{eq:bnd.(a+b)^p}, we have
    \begin{equation}\label{eq:error.int:2}
      \norm[L^p(T)^d]{\GRAD(u-\pT\su)}^p
      \lesssim
      \norm[L^p(T)^d]{\GRAD(u-\eproj[T]{k+1}u)}^p
      + \norm[L^p(T)^d]{\GRAD\pT(\shu[T] - \su)}^p.
    \end{equation}
    Summing \eqref{eq:error.int:1} and \eqref{eq:error.int:2}, and recalling the definition \eqref{def:norm.1} of $\norm[1,p,T]{{\cdot}}$, we obtain
    \begin{multline*}
    \norm[L^p(T)^d]{\GRAD(u-\pT\su)}^p
    + s_T(\su,\su)
    \\
    \lesssim
    \norm[L^p(T)^d]{\GRAD(u-\eproj[T]{k+1}u)}^p + s_T(\shu[T],\shu[T]) + \norm[1,p,T]{\shu[T]-\su}^p.
    \end{multline*}
    The result follows by summing this estimate over $T\in\Th$ and invoking Theorem \ref{thm:Wsp.approx} for the first term in the right-hand side, \eqref{eq:est.sT.shu} for the second, and \eqref{eq:est.error} for the third.

\subsection{Numerical examples}\label{sec:num.ex}

For the sake of completeness, we present here some new numerical examples that demonstrate the orders of convergence achieved by the HHO method in practice.
The test were run using the {\sf hho} software platform\footnote{\emph{Agence pour la Protection des Programmes} deposit number IDDN.FR.001.220005.000.S.P.2016.000.10800}.

\subsubsection{Exponential solution}\label{sec:num.ex:exponential}
We first complete the test cases proposed in Ref.~\cite{Di-Pietro.Droniou:16} by considering the exact solution of Section 4.4 therein for an exponent $p$ strictly smaller than 2.
  More precisely, we solve  on the unit square domain $\Omega=(0,1)^2$ the $p$-Laplace Dirichlet problem with $p=\frac74$ corresponding to the exact solution
$$
  u(\vec{x}) = \exp(x_1 + \pi x_2),
$$
  with suitable right-hand side $f$ inferred from the expression of $u$.
  With this choice, the gradient of $u$ is nonzero, which prevents dealing with singularities.
\begin{figure}
  \centering
  \includegraphics[height=2.5cm]{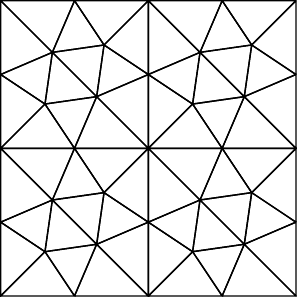}
  \hspace{0.25cm}
  \includegraphics[height=2.5cm]{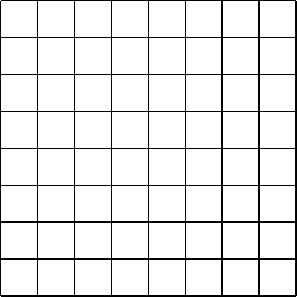}
  \hspace{0.25cm}
  \includegraphics[height=2.5cm]{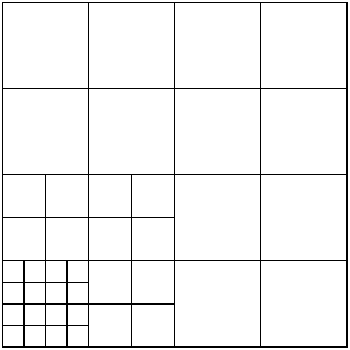}
  \hspace{0.25cm}
  \includegraphics[height=2.5cm]{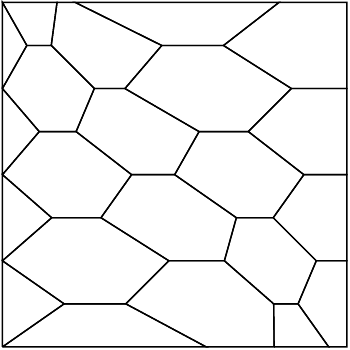}
  \caption{Matching triangular, Cartesian, locally refined and hexagonal mesh families used in the numerical examples of Section \ref{sec:num.ex}.\label{fig:meshes}}
\end{figure}

  We consider the matching triangular, Cartesian, locally refined, and (predominantly) hexagonal mesh families depicted in Figure \ref{fig:meshes} and polynomial degrees ranging from $0$ to $3$.
The three former mesh families have been used in the FVCA5 benchmark\cite{Herbin.Hubert:08}, whereas the latter is taken from Ref.~\cite{Di-Pietro.Lemaire:15}.
The local refinement in the third mesh family has no specific meaning for the problem considered here: its purpose is to demonstrate the seamless treatment of nonconforming interfaces.
%
%
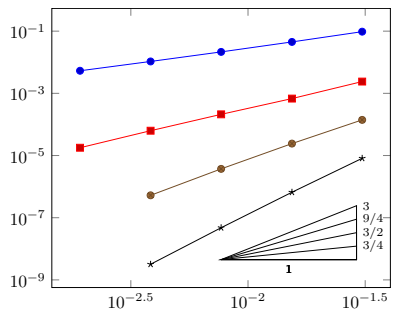
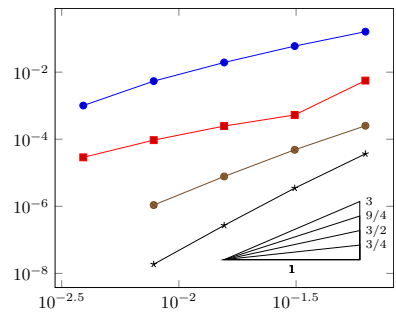
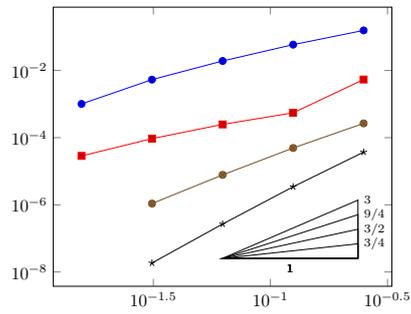
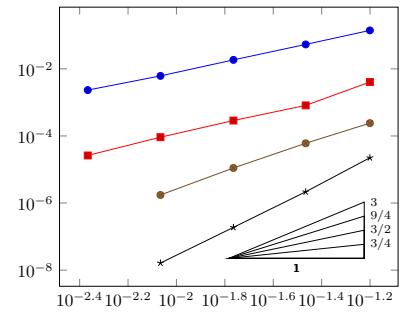
\begin{figure}
  \centering
  \begin{small}
    \tikzexternaldisable
    \tikzexternalenable
  \end{small}
  \begin{minipage}{0.45\textwidth}
    \begin{tikzpicture}[scale=0.65]
      \begin{loglogaxis}[
          legend columns=-1,
          legend to name=conv.legend:2,
          legend style={/tikz/every even column/.append style={column sep=0.35cm}}
        ]
        \addplot table[x=meshsize,y=err_p]{regular-plap2_0_mesh1.dat};
        \addplot table[x=meshsize,y=err_p]{regular-plap2_1_mesh1.dat};
        \addplot table[x=meshsize,y=err_p]{regular-plap2_2_mesh1.dat};
        \addplot table[x=meshsize,y=err_p]{regular-plap2_3_mesh1.dat};
        \logLogSlopeTriangle{0.90}{0.4}{0.1}{3/4}{black};
        \logLogSlopeTriangle{0.90}{0.4}{0.1}{3/2}{black};
        \logLogSlopeTriangle{0.90}{0.4}{0.1}{9/4}{black};
        \logLogSlopeTriangle{0.90}{0.4}{0.1}{3}{black};
        \legend{$k=0$,$k=1$,$k=2$,$k=3$}
      \end{loglogaxis}
    \end{tikzpicture}
    \subcaption{Triangular}
  \end{minipage}
  \begin{minipage}{0.45\textwidth}
    \begin{tikzpicture}[scale=0.65]
      \begin{loglogaxis}
        \addplot table[x=meshsize,y=err_p]{regular-plap2_0_mesh2.dat};
        \addplot table[x=meshsize,y=err_p]{regular-plap2_1_mesh2.dat};
        \addplot table[x=meshsize,y=err_p]{regular-plap2_2_mesh2.dat};
        \addplot table[x=meshsize,y=err_p]{regular-plap2_3_mesh2.dat};
        \logLogSlopeTriangle{0.90}{0.4}{0.1}{3/4}{black};
        \logLogSlopeTriangle{0.90}{0.4}{0.1}{3/2}{black};
        \logLogSlopeTriangle{0.90}{0.4}{0.1}{9/4}{black};
        \logLogSlopeTriangle{0.90}{0.4}{0.1}{3}{black};
      \end{loglogaxis}
    \end{tikzpicture}
    \subcaption{Cartesian}
  \end{minipage}
  \vspace{0.5cm}\\
  \begin{minipage}{0.45\textwidth}
    \begin{tikzpicture}[scale=0.65]
      \begin{loglogaxis}
        \addplot table[x=meshsize,y=err_p]{regular-plap2_0_mesh3.dat};
        \addplot table[x=meshsize,y=err_p]{regular-plap2_1_mesh3.dat};
        \addplot table[x=meshsize,y=err_p]{regular-plap2_2_mesh3.dat};
        \addplot table[x=meshsize,y=err_p]{regular-plap2_3_mesh3.dat};
        \logLogSlopeTriangle{0.90}{0.4}{0.1}{3/4}{black};
        \logLogSlopeTriangle{0.90}{0.4}{0.1}{3/2}{black};
        \logLogSlopeTriangle{0.90}{0.4}{0.1}{9/4}{black};
        \logLogSlopeTriangle{0.90}{0.4}{0.1}{3}{black};
      \end{loglogaxis}
    \end{tikzpicture}
    \subcaption{Loc. ref.}
  \end{minipage}
  \begin{minipage}{0.45\textwidth}
    \begin{tikzpicture}[scale=0.65]
      \begin{loglogaxis}
        \addplot table[x=meshsize,y=err_p]{regular-plap2_0_pi6_tiltedhexagonal.dat};
        \addplot table[x=meshsize,y=err_p]{regular-plap2_1_pi6_tiltedhexagonal.dat};
        \addplot table[x=meshsize,y=err_p]{regular-plap2_2_pi6_tiltedhexagonal.dat};
        \addplot table[x=meshsize,y=err_p]{regular-plap2_3_pi6_tiltedhexagonal.dat};
        \logLogSlopeTriangle{0.90}{0.4}{0.1}{3/4}{black};
        \logLogSlopeTriangle{0.90}{0.4}{0.1}{3/2}{black};
        \logLogSlopeTriangle{0.90}{0.4}{0.1}{9/4}{black};
        \logLogSlopeTriangle{0.90}{0.4}{0.1}{3}{black};
      \end{loglogaxis}
    \end{tikzpicture}
    \subcaption{Hexagonal}
  \end{minipage}
  \caption{$\norm[1,p,h]{\Ih u-\su[h]}$ versus $h$ for the test of Section~\ref{sec:num.ex:exponential} and the mesh families of Figure \ref{fig:meshes} with $p=\frac{7}{4}$.
    The slopes represent the orders of convergence expected from Theorem \ref{thm:est.error}, i.e. $\frac{3}{4}(k+1)$ for $k\in\{0,\ldots,3\}$ (resp. blue dots, red squares, brown dots, black stars).\label{fig:results:2}}
\end{figure}

We report in Figure \ref{fig:results:2} the error $\norm[1,p,h]{\Ih u-\su[h]}$ versus the meshsize $h$.
The observed orders of convergence seem to suggest that our estimate~\eqref{eq:est.error} is sharp.
For $k=0$, superconvergence is observed on the Cartesian mesh family and, to a lesser extent, on the locally refined mesh family. This kind of superconvergence phenomena have already been observed in the past for the Poisson problem corresponding to $p=2$ (to this date, a theoretical investigation is still not available).

\subsubsection{Trigonometric solution}\label{sec:num.ex:trigonometric}

The test case of the previous section had already been solved in Ref.~\cite{Di-Pietro.Droniou:16} for different vales of $p$ greater or equal than 2. We therefore consider here a different manufactured solution.
We solve on the unit square domain $\Omega=(0,1)^2$ the homogeneous $p$-Laplace Dirichlet problem corresponding to the exact solution
$$
u(\vec{x}) = \sin(\pi x_1)\sin(\pi x_2),
$$
with $p\in\{2,3,4\}$ and source term inferred from $u$ (cf. \eqref{choice:plap} for the expression of $\bfa$ in this case).
We consider the same mesh families and polynomial orders as in the previous section.

We report in Figure \ref{fig:results} the error $\norm[1,p,h]{\Ih u-\su[h]}$ versus the meshsize $h$.
  From the leftmost column, we see that the error estimates are sharp for $p=2$, which confirms the results of Ref.~\cite{Di-Pietro.Ern.ea:14} (a known superconvergence phenomenon is observed on the Cartesian mesh for $k=0$).
  For $p=3,4$, better orders of convergence than the asymptotic ones (cf. Remark \ref{rem:order.of.conv}) are observed in most of the cases.
  One possible explanation is that the lowest-order terms in the right-hand side of \eqref{eq:est.error} are not yet dominant for the specific problem data and mesh at hand.
  Another possibility is that compensations occur among lowest-order terms that are separately estimated in the proof of Theorem \ref{thm:est.error}.
  For $k=3$ and $p=3$, the observed orders of convergence in the last refinement steps are inferior to the predicted value for smooth solutions, which can likely be ascribed to the violation of the regularity assumption on $\bfa(\cdot,\GRAD u)$ (cf. Theorem \ref{thm:est.error}), due to the lack of smoothness of $\bfa$ for that $p$.

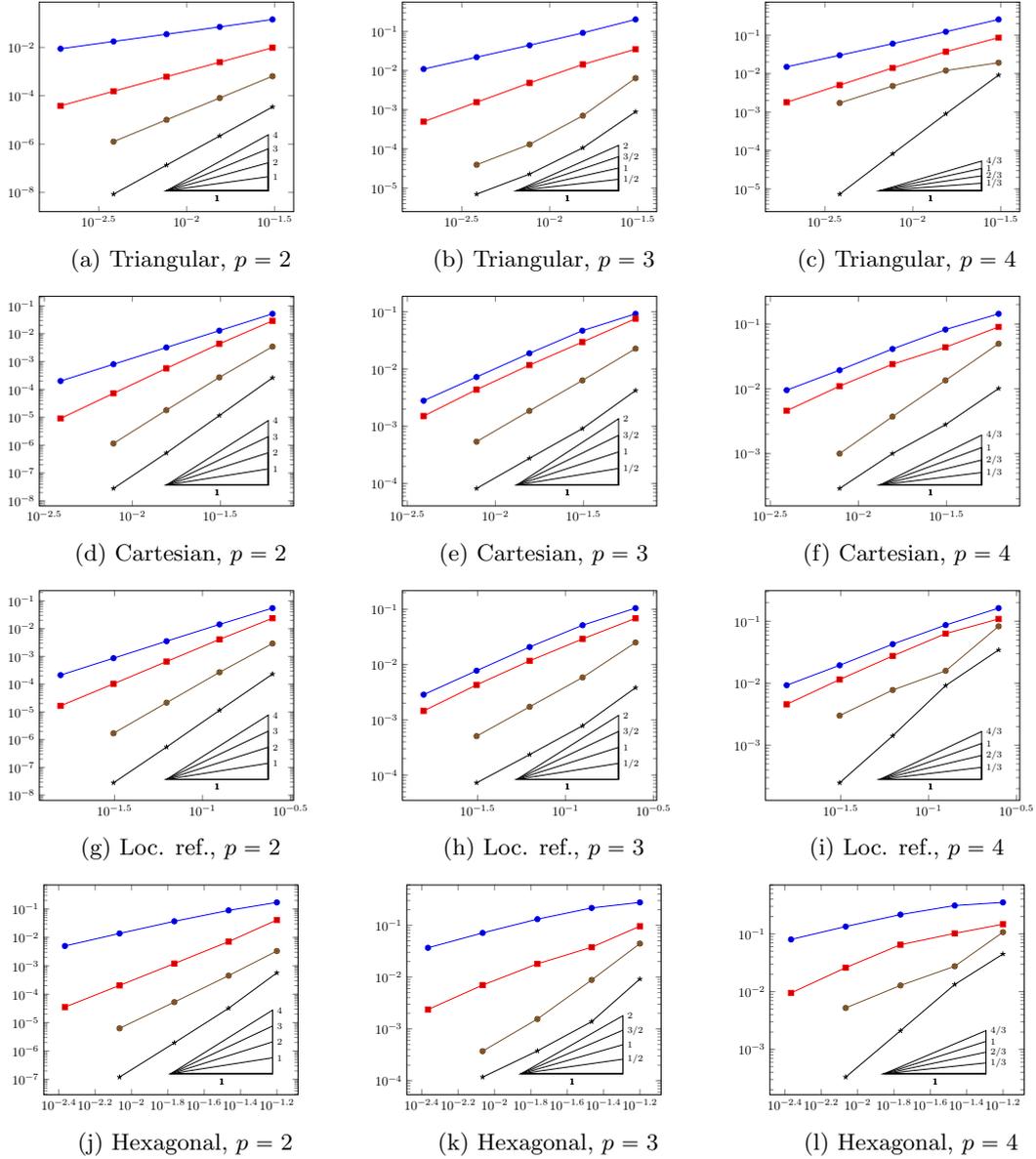
\begin{figure}
  \centering
  \begin{small}
    \tikzexternaldisable
    \tikzexternalenable
  \end{small}
  \begin{minipage}{0.32\textwidth}
    \begin{tikzpicture}[scale=0.50]
      \begin{loglogaxis}[
          legend columns=-1,
          legend to name=conv.legend,
          legend style={/tikz/every even column/.append style={column sep=0.35cm}}
        ]
        \addplot table[x=meshsize,y=errp2]{pge2-plap_0_mesh1_ocv_nous.dat};
        \addplot table[x=meshsize,y=errp2]{pge2-plap_1_mesh1_ocv_nous.dat};
        \addplot table[x=meshsize,y=errp2]{pge2-plap_2_mesh1_ocv_nous.dat};
        \addplot table[x=meshsize,y=errp2]{pge2-plap_3_mesh1_ocv_nous.dat};
        \logLogSlopeTriangle{0.90}{0.4}{0.1}{1}{black};
        \logLogSlopeTriangle{0.90}{0.4}{0.1}{2}{black};
        \logLogSlopeTriangle{0.90}{0.4}{0.1}{3}{black};
        \logLogSlopeTriangle{0.90}{0.4}{0.1}{4}{black};
        \legend{$k=0$,$k=1$,$k=2$,$k=3$}
      \end{loglogaxis}
    \end{tikzpicture}
    \subcaption{Triangular, $p=2$}
  \end{minipage}  
  \begin{minipage}{0.32\textwidth}  
    \begin{tikzpicture}[scale=0.50]
      \begin{loglogaxis}
        \addplot table[x=meshsize,y=errp3]{pge2-plap_0_mesh1_ocv_nous.dat};
        \addplot table[x=meshsize,y=errp3]{pge2-plap_1_mesh1_ocv_nous.dat};
        \addplot table[x=meshsize,y=errp3]{pge2-plap_2_mesh1_ocv_nous.dat};
        \addplot table[x=meshsize,y=errp3]{pge2-plap_3_mesh1_ocv_nous.dat};
        \logLogSlopeTriangle{0.85}{0.4}{0.1}{1/2}{black};
        \logLogSlopeTriangle{0.85}{0.4}{0.1}{1}{black};
        \logLogSlopeTriangle{0.85}{0.4}{0.1}{3/2}{black};
        \logLogSlopeTriangle{0.85}{0.4}{0.1}{2}{black};        
      \end{loglogaxis}
    \end{tikzpicture}
    \subcaption{Triangular, $p=3$}    
  \end{minipage}  
  \begin{minipage}{0.32\textwidth}  
    \begin{tikzpicture}[scale=0.50]
      \begin{loglogaxis}
        \addplot table[x=meshsize,y=errp4]{pge2-plap_0_mesh1_ocv_nous.dat};
        \addplot table[x=meshsize,y=errp4]{pge2-plap_1_mesh1_ocv_nous.dat};
        \addplot table[x=meshsize,y=errp4]{pge2-plap_2_mesh1_ocv_nous.dat};
        \addplot table[x=meshsize,y=errp4]{pge2-plap_3_mesh1_ocv_nous.dat};
        \logLogSlopeTriangle{0.85}{0.4}{0.1}{1/3}{black};
        \logLogSlopeTriangle{0.85}{0.4}{0.1}{2/3}{black};
        \logLogSlopeTriangle{0.85}{0.4}{0.1}{1}{black};
        \logLogSlopeTriangle{0.85}{0.4}{0.1}{4/3}{black};        
      \end{loglogaxis}
    \end{tikzpicture}
    \subcaption{Triangular, $p=4$}    
  \end{minipage}
  \vspace{0.25cm} \\
  \begin{minipage}{0.32\textwidth}
    \begin{tikzpicture}[scale=0.50]
      \begin{loglogaxis}
        \addplot table[x=meshsize,y=errp2]{pge2-plap_0_mesh2_ocv_nous.dat};
        \addplot table[x=meshsize,y=errp2]{pge2-plap_1_mesh2_ocv_nous.dat};
        \addplot table[x=meshsize,y=errp2]{pge2-plap_2_mesh2_ocv_nous.dat};
        \addplot table[x=meshsize,y=errp2]{pge2-plap_3_mesh2_ocv_nous.dat};
        \logLogSlopeTriangle{0.90}{0.4}{0.1}{1}{black};
        \logLogSlopeTriangle{0.90}{0.4}{0.1}{2}{black};
        \logLogSlopeTriangle{0.90}{0.4}{0.1}{3}{black};
        \logLogSlopeTriangle{0.90}{0.4}{0.1}{4}{black};
      \end{loglogaxis}
    \end{tikzpicture}
    \subcaption{Cartesian, $p=2$}
  \end{minipage}  
  \begin{minipage}{0.32\textwidth}  
    \begin{tikzpicture}[scale=0.50]
      \begin{loglogaxis}
        \addplot table[x=meshsize,y=errp3]{pge2-plap_0_mesh2_ocv_nous.dat};
        \addplot table[x=meshsize,y=errp3]{pge2-plap_1_mesh2_ocv_nous.dat};
        \addplot table[x=meshsize,y=errp3]{pge2-plap_2_mesh2_ocv_nous.dat};
        \addplot table[x=meshsize,y=errp3]{pge2-plap_3_mesh2_ocv_nous.dat};
        \logLogSlopeTriangle{0.85}{0.4}{0.1}{1/2}{black};
        \logLogSlopeTriangle{0.85}{0.4}{0.1}{1}{black};
        \logLogSlopeTriangle{0.85}{0.4}{0.1}{3/2}{black};
        \logLogSlopeTriangle{0.85}{0.4}{0.1}{2}{black};        
      \end{loglogaxis}
    \end{tikzpicture}
    \subcaption{Cartesian, $p=3$}    
  \end{minipage}  
  \begin{minipage}{0.32\textwidth}  
    \begin{tikzpicture}[scale=0.50]
      \begin{loglogaxis}
        \addplot table[x=meshsize,y=errp4]{pge2-plap_0_mesh2_ocv_nous.dat};
        \addplot table[x=meshsize,y=errp4]{pge2-plap_1_mesh2_ocv_nous.dat};
        \addplot table[x=meshsize,y=errp4]{pge2-plap_2_mesh2_ocv_nous.dat};
        \addplot table[x=meshsize,y=errp4]{pge2-plap_3_mesh2_ocv_nous.dat};
        \logLogSlopeTriangle{0.85}{0.4}{0.1}{1/3}{black};
        \logLogSlopeTriangle{0.85}{0.4}{0.1}{2/3}{black};
        \logLogSlopeTriangle{0.85}{0.4}{0.1}{1}{black};
        \logLogSlopeTriangle{0.85}{0.4}{0.1}{4/3}{black};        
      \end{loglogaxis}
    \end{tikzpicture}
    \subcaption{Cartesian, $p=4$}    
  \end{minipage}
  \vspace{0.25cm} \\
  \begin{minipage}{0.32\textwidth}
    \begin{tikzpicture}[scale=0.50]
      \begin{loglogaxis}
        \addplot table[x=meshsize,y=errp2]{pge2-plap_0_mesh3_ocv_nous.dat};
        \addplot table[x=meshsize,y=errp2]{pge2-plap_1_mesh3_ocv_nous.dat};
        \addplot table[x=meshsize,y=errp2]{pge2-plap_2_mesh3_ocv_nous.dat};
        \addplot table[x=meshsize,y=errp2]{pge2-plap_3_mesh3_ocv_nous.dat};
        \logLogSlopeTriangle{0.90}{0.4}{0.1}{1}{black};
        \logLogSlopeTriangle{0.90}{0.4}{0.1}{2}{black};
        \logLogSlopeTriangle{0.90}{0.4}{0.1}{3}{black};
        \logLogSlopeTriangle{0.90}{0.4}{0.1}{4}{black};
      \end{loglogaxis}
    \end{tikzpicture}
    \subcaption{Loc. ref., $p=2$}
  \end{minipage}  
  \begin{minipage}{0.32\textwidth}  
    \begin{tikzpicture}[scale=0.50]
      \begin{loglogaxis}
        \addplot table[x=meshsize,y=errp3]{pge2-plap_0_mesh3_ocv_nous.dat};
        \addplot table[x=meshsize,y=errp3]{pge2-plap_1_mesh3_ocv_nous.dat};
        \addplot table[x=meshsize,y=errp3]{pge2-plap_2_mesh3_ocv_nous.dat};
        \addplot table[x=meshsize,y=errp3]{pge2-plap_3_mesh3_ocv_nous.dat};
        \logLogSlopeTriangle{0.85}{0.4}{0.1}{1/2}{black};
        \logLogSlopeTriangle{0.85}{0.4}{0.1}{1}{black};
        \logLogSlopeTriangle{0.85}{0.4}{0.1}{3/2}{black};
        \logLogSlopeTriangle{0.85}{0.4}{0.1}{2}{black};        
      \end{loglogaxis}
    \end{tikzpicture}
    \subcaption{Loc. ref., $p=3$}    
  \end{minipage}  
  \begin{minipage}{0.32\textwidth}  
    \begin{tikzpicture}[scale=0.50]
      \begin{loglogaxis}
        \addplot table[x=meshsize,y=errp4]{pge2-plap_0_mesh3_ocv_nous.dat};
        \addplot table[x=meshsize,y=errp4]{pge2-plap_1_mesh3_ocv_nous.dat};
        \addplot table[x=meshsize,y=errp4]{pge2-plap_2_mesh3_ocv_nous.dat};
        \addplot table[x=meshsize,y=errp4]{pge2-plap_3_mesh3_ocv_nous.dat};
        \logLogSlopeTriangle{0.85}{0.4}{0.1}{1/3}{black};
        \logLogSlopeTriangle{0.85}{0.4}{0.1}{2/3}{black};
        \logLogSlopeTriangle{0.85}{0.4}{0.1}{1}{black};
        \logLogSlopeTriangle{0.85}{0.4}{0.1}{4/3}{black};        
      \end{loglogaxis}
    \end{tikzpicture}
    \subcaption{Loc. ref., $p=4$}    
  \end{minipage}
  \vspace{0.25cm} \\
  \begin{minipage}{0.32\textwidth}
    \begin{tikzpicture}[scale=0.50]
      \begin{loglogaxis}
        \addplot table[x=meshsize,y=errp2]{pge2-plap_0_pi6_tiltedhexagonal_ocv_nous.dat};
        \addplot table[x=meshsize,y=errp2]{pge2-plap_1_pi6_tiltedhexagonal_ocv_nous.dat};
        \addplot table[x=meshsize,y=errp2]{pge2-plap_2_pi6_tiltedhexagonal_ocv_nous.dat};
        \addplot table[x=meshsize,y=errp2]{pge2-plap_3_pi6_tiltedhexagonal_ocv_nous.dat};
        \logLogSlopeTriangle{0.90}{0.4}{0.1}{1}{black};
        \logLogSlopeTriangle{0.90}{0.4}{0.1}{2}{black};
        \logLogSlopeTriangle{0.90}{0.4}{0.1}{3}{black};
        \logLogSlopeTriangle{0.90}{0.4}{0.1}{4}{black};
      \end{loglogaxis}
    \end{tikzpicture}
    \subcaption{Hexagonal, $p=2$}
  \end{minipage}  
  \begin{minipage}{0.32\textwidth}  
    \begin{tikzpicture}[scale=0.50]
      \begin{loglogaxis}
        \addplot table[x=meshsize,y=errp3]{pge2-plap_0_pi6_tiltedhexagonal_ocv_nous.dat};
        \addplot table[x=meshsize,y=errp3]{pge2-plap_1_pi6_tiltedhexagonal_ocv_nous.dat};
        \addplot table[x=meshsize,y=errp3]{pge2-plap_2_pi6_tiltedhexagonal_ocv_nous.dat};
        \addplot table[x=meshsize,y=errp3]{pge2-plap_3_pi6_tiltedhexagonal_ocv_nous.dat};
        \logLogSlopeTriangle{0.85}{0.4}{0.1}{1/2}{black};
        \logLogSlopeTriangle{0.85}{0.4}{0.1}{1}{black};
        \logLogSlopeTriangle{0.85}{0.4}{0.1}{3/2}{black};
        \logLogSlopeTriangle{0.85}{0.4}{0.1}{2}{black};        
      \end{loglogaxis}
    \end{tikzpicture}
    \subcaption{Hexagonal, $p=3$}    
  \end{minipage}  
  \begin{minipage}{0.32\textwidth}  
    \begin{tikzpicture}[scale=0.50]
      \begin{loglogaxis}
        \addplot table[x=meshsize,y=errp4]{pge2-plap_0_pi6_tiltedhexagonal_ocv_nous.dat};
        \addplot table[x=meshsize,y=errp4]{pge2-plap_1_pi6_tiltedhexagonal_ocv_nous.dat};
        \addplot table[x=meshsize,y=errp4]{pge2-plap_2_pi6_tiltedhexagonal_ocv_nous.dat};
        \addplot table[x=meshsize,y=errp4]{pge2-plap_3_pi6_tiltedhexagonal_ocv_nous.dat};
        \logLogSlopeTriangle{0.85}{0.4}{0.1}{1/3}{black};
        \logLogSlopeTriangle{0.85}{0.4}{0.1}{2/3}{black};
        \logLogSlopeTriangle{0.85}{0.4}{0.1}{1}{black};
        \logLogSlopeTriangle{0.85}{0.4}{0.1}{4/3}{black};        
      \end{loglogaxis}
    \end{tikzpicture}
    \subcaption{Hexagonal, $p=4$}    
  \end{minipage}
  \caption{$\norm[1,p,h]{\Ih u-\su[h]}$ versus $h$ for the test of Section~\ref{sec:num.ex:trigonometric} and the mesh families of Figure \ref{fig:meshes}.
    The slopes represent the orders of convergence expected from Theorem \ref{thm:est.error}, i.e. $\frac{k+1}{p-1}$ for $k\in\{0,\ldots,3\}$ (resp. blue dots, red squares, brown dots, black stars) and $p\in\{2,3,4\}$.\label{fig:results}}
\end{figure}
%
%
\begin{figure}
  \centering
  \begin{small}
    \tikzexternaldisable
    \tikzexternalenable
  \end{small}
  \begin{minipage}{0.45\textwidth}
    \begin{tikzpicture}[scale=0.65]
      \begin{loglogaxis}[
          legend columns=-1,
          legend to name=conv.legend:2,
          legend style={/tikz/every even column/.append style={column sep=0.35cm}},
          ymin=1e-6
        ]
        \addplot table[x=meshsize,y=err_p]{singular-plap2_0_mesh1.dat};
        \addplot table[x=meshsize,y=err_p]{singular-plap2_1_mesh1.dat};
        \addplot table[x=meshsize,y=err_p]{singular-plap2_2_mesh1.dat};
        \addplot table[x=meshsize,y=err_p]{singular-plap2_3_mesh1.dat};
        \logLogSlopeTriangle{0.90}{0.4}{0.1}{3/4}{black};
        \logLogSlopeTriangle{0.90}{0.4}{0.1}{3/2}{black};
        \logLogSlopeTriangle{0.90}{0.4}{0.1}{9/4}{black};
        \logLogSlopeTriangle{0.90}{0.4}{0.1}{3}{black};
        \legend{$k=0$,$k=1$,$k=2$,$k=3$}
      \end{loglogaxis}
    \end{tikzpicture}
    \subcaption{Triangular}
  \end{minipage}
  \begin{minipage}{0.45\textwidth}
    \begin{tikzpicture}[scale=0.65]
      \begin{loglogaxis}[ymin=1e-5]
        \addplot table[x=meshsize,y=err_p]{singular-plap2_0_mesh2.dat};
        \addplot table[x=meshsize,y=err_p]{singular-plap2_1_mesh2.dat};
        \addplot table[x=meshsize,y=err_p]{singular-plap2_2_mesh2.dat};
        \addplot table[x=meshsize,y=err_p]{singular-plap2_3_mesh2.dat};
        \logLogSlopeTriangle{0.90}{0.4}{0.1}{3/4}{black};
        \logLogSlopeTriangle{0.90}{0.4}{0.1}{3/2}{black};
        \logLogSlopeTriangle{0.90}{0.4}{0.1}{9/4}{black};
        \logLogSlopeTriangle{0.90}{0.4}{0.1}{3}{black};
      \end{loglogaxis}
    \end{tikzpicture}
    \subcaption{Cartesian}
  \end{minipage}
  \vspace{0.5cm}\\
  \begin{minipage}{0.45\textwidth}
    \begin{tikzpicture}[scale=0.65]
      \begin{loglogaxis}[ymin=1e-5]
        \addplot table[x=meshsize,y=err_p]{singular-plap2_0_mesh3.dat};
        \addplot table[x=meshsize,y=err_p]{singular-plap2_1_mesh3.dat};
        \addplot table[x=meshsize,y=err_p]{singular-plap2_2_mesh3.dat};
        \addplot table[x=meshsize,y=err_p]{singular-plap2_3_mesh3.dat};
        \logLogSlopeTriangle{0.90}{0.4}{0.1}{3/4}{black};
        \logLogSlopeTriangle{0.90}{0.4}{0.1}{3/2}{black};
        \logLogSlopeTriangle{0.90}{0.4}{0.1}{9/4}{black};
        \logLogSlopeTriangle{0.90}{0.4}{0.1}{3}{black};
      \end{loglogaxis}
    \end{tikzpicture}
    \subcaption{Loc. ref.}
  \end{minipage}
  \begin{minipage}{0.45\textwidth}
    \begin{tikzpicture}[scale=0.65]
      \begin{loglogaxis}[ymin=1e-5]
        \addplot table[x=meshsize,y=err_p]{singular-plap2_0_pi6_tiltedhexagonal.dat};
        \addplot table[x=meshsize,y=err_p]{singular-plap2_1_pi6_tiltedhexagonal.dat};
        \addplot table[x=meshsize,y=err_p]{singular-plap2_2_pi6_tiltedhexagonal.dat};
        \addplot table[x=meshsize,y=err_p]{singular-plap2_3_pi6_tiltedhexagonal.dat};
        \logLogSlopeTriangle{0.90}{0.4}{0.1}{3/4}{black};
        \logLogSlopeTriangle{0.90}{0.4}{0.1}{3/2}{black};
        \logLogSlopeTriangle{0.90}{0.4}{0.1}{9/4}{black};
        \logLogSlopeTriangle{0.90}{0.4}{0.1}{3}{black};
      \end{loglogaxis}
    \end{tikzpicture}
    \subcaption{Hexagonal}
  \end{minipage}
  \caption{$\norm[1,p,h]{\Ih u-\su[h]}$ versus $h$ for the test of Section~\ref{sec:num.ex:trigonometric} and the mesh families of Figure \ref{fig:meshes} with $p=\frac{7}{4}$.
    The slopes represent the orders of convergence expected from Theorem \ref{thm:est.error}, i.e. $\frac{3}{4}(k+1)$ for $k\in\{0,\ldots,3\}$ (resp. blue dots, red squares, brown dots, black stars). In this case, the order of convergence is limited by the regularity of $\bfa(\cdot,\GRAD u)$.\label{fig:results:3}}
\end{figure}
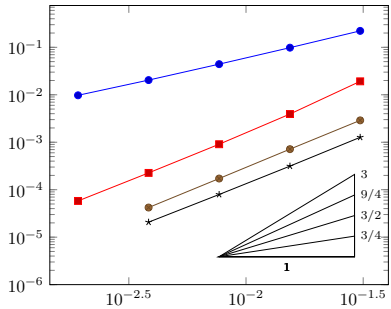
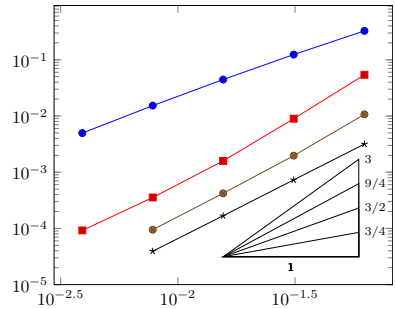
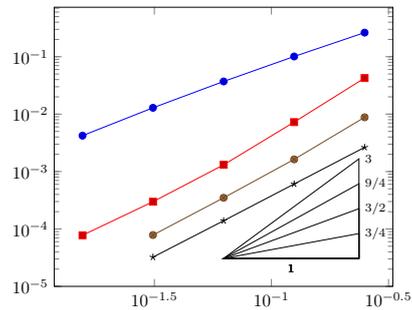
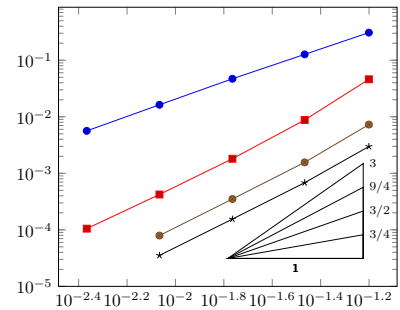

For the sake of completeness, we consider also in this case the exponent $p=\frac74$. Unlike in the previous section, the regularity 
required by the error estimates in Theorem~\ref{thm:est.error} does not seem to hold for the exact solution.
This can be clearly seen in Figure \ref{fig:results:3}, where the observed convergence rate seems optimal for $k\in\{0,1\}$, while it stagnates for $k\in\{2,3\}$.
For this value of $p$, the derivatives of $\vec{\xi}\to|\vec{\xi}|^{p-2}\vec{\xi}$ are singular
at $\vec{\xi}=\vec{0}$, which prevents $\bfa(\cdot,\GRAD u)$ from having the
$W^{k+1,p'}$ regularity required in \eqref{def:Eh.1}.

%
%
\appendix

\section{Inequalities involving the Leray--Lions operator}\label{app:llop}

This section collects inequalities involving the Leray--Lions operator adapted from Ref.~\cite{gdm}.

\begin{lemma}\label{lem:tech.1}
  Assume \eqref{hyp:ag}, \eqref{hyp:alip}, and $p\le 2$. Then, for a.e. $\vec{x}\in\Omega$
  and all $(\vec{\xi},\vec{\eta})\in\Real^d\times\Real^d$,
  \begin{equation}\label{alip:p2}
    |\bfa(\vec{x},\vec{\xi})-\bfa(\vec{x},\vec{\eta})|\le (2\lipa+2^{p-1}\upa+\upa)|\vec{\xi}-\vec{\eta}|^{p-1}.
  \end{equation}
\end{lemma}

\begin{proof}
  Let $r>0$.
  If $|\vec{\xi}|\ge r$ and $|\vec{\eta}|\ge r$ then, using \eqref{hyp:alip} and $p-2\le 0$, we have
  \begin{equation}\label{tech.proof:1}
    |\bfa(\vec{x},\vec{\xi})-\bfa(\vec{x},\vec{\eta})|\le \lipa
    |\vec{\xi}-\vec{\eta}|(|\vec{\xi}|^{p-2}+|\vec{\eta}|^{p-2})
    \le 2\lipa r^{p-2}|\vec{\xi}-\vec{\eta}|.
  \end{equation}
  Otherwise, assume for example that $|\vec{\eta}|<r$. Then $|\vec{\xi}|\le |\vec{\xi}-\vec{\eta}|+r$
  and thus, owing to \eqref{hyp:ag},
  \begin{align}
    |\bfa(\vec{x},\vec{\xi})-\bfa(\vec{x},\vec{\eta})|\le{}&
    |\bfa(\vec{x},\vec{\xi})-\bfa(\vec{x},\vec{0})|+|\bfa(\vec{x},\vec{0})-\bfa(\vec{x},\vec{\eta})|\nonumber\\
    \le{}&\upa (|\vec{\xi}|^{p-1}+|\vec{\eta}|^{p-1})\nonumber\\
    \le{}& 
    \upa(|\vec{\xi}-\vec{\eta}|+r)^{p-1}+\upa r^{p-1}.
    \label{tech.proof:2}
  \end{align}
  Combining \eqref{tech.proof:1} and \eqref{tech.proof:2} shows that, in either case,
  \[
  |\bfa(\vec{x},\vec{\xi})-\bfa(\vec{x},\vec{\eta})|\le
  2\lipa r^{p-2}|\vec{\xi}-\vec{\eta}|+
  \upa(|\vec{\xi}-\vec{\eta}|+r)^{p-1}+\upa r^{p-1}.
  \]
  Taking $r=|\vec{\xi}-\vec{\eta}|$ concludes the proof of \eqref{alip:p2}.
\end{proof}

\begin{lemma}\label{lem:tech.2}
  Under Assumption \eqref{hyp:am} we have, for a.e. $\vec{x}\in\Omega$ and all $(\vec{\xi},\vec{\eta})\in\Real^d\times\Real^d$,
  \begin{asparaitem}
  \item If $p<2$,
    \begin{equation}
      \label{amon:p2m}
      |\vec{\xi}-\vec{\eta}|^p\le \mona^{-\frac{p}{2}}2^{(p-1)\frac{2-p}{2}}\Big(
      [\bfa(\vec{x},\vec{\xi})-\bfa(\vec{x},\vec{\eta})]\cdot[\vec{\xi}-\vec{\eta}]\Big)^{\frac{p}{2}}
      \Big(|\vec{\xi}|^p+|\vec{\eta}|^p\Big)^{\frac{2-p}{2}};
    \end{equation}
  \item If $p\ge 2$,
    \begin{equation}
      \label{amon:p2p}
      |\vec{\xi}-\vec{\eta}|^p\le \mona^{-1}
      [\bfa(\vec{x},\vec{\xi})-\bfa(\vec{x},\vec{\eta})]\cdot[\vec{\xi}-\vec{\eta}].
    \end{equation}
  \end{asparaitem}
\end{lemma}

\begin{proof}
  Estimate \eqref{amon:p2m} is obtained by raising \eqref{hyp:am} to the power $p/2$ and using $(|\vec{\xi}|+|\vec{\eta}|)^p
  \le 2^{p-1}(|\vec{\xi}|^p+|\vec{\eta}|^p)$.
  To prove \eqref{amon:p2p}, we simply write $|\vec{\xi}-\vec{\eta}|^p\le |\vec{\xi}-\vec{\eta}|^2
  (|\vec{\xi}|+|\vec{\eta}|)^{p-2}$. \end{proof}

\begin{remark} The (real-valued) mapping $a:t\mapsto |t|^{p-2}t$ corresponds to the $p$-Laplace operator
  in dimension 1, and it therefore satisfies \eqref{hyp:am}. Hence, by Lemma \ref{lem:tech.2},
  \begin{align}\label{eq:mon.sT.1}
    &\text{If $p< 2$:}\quad |t-r|^p\le C\left( \left[|t|^{p-2}t-|r|^{p-2}r\right] [t-r] \right)^{\frac{p}{2}}
    \left(|t|^p+|r|^p\right)^{\frac{2-p}{p}},\\
    \label{eq:mon.sT.2}
    &\text{If $p\ge 2$:}\quad |t-r|^p\le C\left[|t|^{p-2}t-|r|^{p-2}r\right][t-r],
  \end{align}
  where $C$ depends only on $p$.
\end{remark}


\section*{Acknowledgment}
This work was partially supported by Agence Nationale de la Recherche project HHOMM (ANR-15-CE40-0005) and by the Australian Research Council's Discovery Projects funding scheme (project number DP170100605).


\begin{footnotesize}
  \bibliographystyle{plain}
  \bibliography{eproj}
\end{footnotesize}

\end{document}